\documentclass[12pt]{article}

\usepackage[utf8]{inputenc}
\usepackage{mathrsfs}
\usepackage{amsmath}
\usepackage{amsthm}
\usepackage{amssymb}
\usepackage[T1]{fontenc}
\usepackage[utf8]{inputenc}
\usepackage[french,english]{babel}
\usepackage{amsmath,amsfonts,amssymb}
\usepackage{graphicx}
\usepackage[colorinlistoftodos]{todonotes}
\usepackage{float}
\usepackage{here}
\usepackage{fancyhdr}
\usepackage{color}
\usepackage{mathptmx}
\usepackage{lipsum}% just to generate filler text
\usepackage{setspace}
\usepackage{geometry}
\usepackage{amsthm}
\usepackage{multirow}
\usepackage{multicol}
\usepackage{enumitem}
\usepackage{graphicx}
\usepackage{comment} 
\usetikzlibrary{arrows.meta}
\usepackage{eso-pic}	% Nécessaire pour mettre des images en arrière plan
\usepackage{array}		

\usepackage{tikz}

\usepackage[resetlabels,labeled]{multibib}

\newtheorem{theorem}{Theorem}
\newtheorem{lemma}{Lemma}
\newtheorem{conjecture}{Conjecture}
\newtheorem{observation}{Observation}

\newtheorem{definition}{Definition}

\newtheorem*{conjecture-non}{Conjecture}

\title{Labeled Packing of Cycles and Circuits}
\author{Alice Joffard, Hamamache Kheddouci}
\date{}

\begin{document}
\maketitle

\begin{comment}

In 2013, Duchêne, Kheddouci, Nowakowski and Tahraoui introduced a labeled version of the graph packing problem. It led to the introduction of a new parameter for graphs, the k-labeled packing number $\lambda^k$. This parameter corresponds to the maximum number of labels we can assign to the vertices of the graph, such that we will be able to create a packing of k copies of the graph, while conserving the labels of the vertices.

The authors intensively studied the labeled packing of cycles, and, among other results, they conjectured that for every cycle C_n of order n=2k+x, with k>1 and 0<x<2k, the value of $\lambda^k(C_n)$ was 2 if x was 1 and k was even, and x+2 otherwise.

In this paper, we disprove this conjecture by giving a counter example. We however prove that it gives a valid lower bound, and we give sufficient conditions for the upper bound to hold.

We then give some similar results for the labeled packing of circuits.

\end{comment}

\begin{abstract} In 2013, Duchêne, Kheddouci, Nowakowski and Tahraoui \cite{duchene2013labeled,tahraoui2012coloring} introduced a labeled version of the graph packing problem. It led to the introduction of a new parameter for graphs, the $k$-labeled packing number $\lambda^k$. This parameter corresponds to the maximum number of labels we can assign to the vertices of the graph, such that we will be able to create a packing of $k$ copies of the graph, while conserving the labels of the vertices.
\\ \indent The authors intensively studied the labeled packing of cycles, and, among other results, they conjectured that for every cycle $C_n$ of order $n=2k+x$, with $k\geq 2$ and $1\leq x \leq 2k-1$, the value of $\lambda^k(C_n)$ was $2$ if $x$ was $1$ and $k$ was even, and $x+2$ otherwise.
\\ \indent In this paper, we disprove this conjecture by giving a counter example. We however prove that it gives a valid lower bound, and we give sufficient conditions for the upper bound to hold.
\\ \indent We then give some similar results for the labeled packing of circuits.
\\ \hspace{1cm}\vspace{0.3cm} 
\\ \textbf{Key words:} Packing of graphs, Labeled packing, Cycles, Circuits.
\end{abstract}

\section{Definitions and notations}

\indent We will work with simple graphs, with no multiple edges nor loops. For an undirected graph $G=(V(G),E(G))$, $V(G)$ will represent the set of vertices of $G$, $E(G)$ will represent the set of edges of $G$. For a directed graph $G=(V(G),E(G))$, $V(G)$ will represent the set of vertices of $G$, $E(G)$ will represent the set of arcs of $G$. The \textit{order} $n$ of a graph $G$ will represent its number of vertices, its \textit{size} $m$ will represent its number of edges or arcs, and $\Delta(G)$ will represent its maximal degree. 
\\ \indent $K_n$ will be the complete undirected graph of order $n$, $K_{p,q}$ the complete bipartite graph of partition sizes $p$ and $q$, $P_n$ the path of order $n$, $C_n$ the cycle of order $n$, and $S_n$ the star of order $n$. $\overleftrightarrow{K_n}$ will be the \textit{complete directed graph} of order $n$, meaning that its set of arcs contains all the possible arcs from one of its vertices to another, and $\overrightarrow{C_n}$  will be the \textit{circuit} of order $n$, that is the cyclic orientation of the cycle $C_n$.
\\ \indent $G\cup H$, the \textit{union} of $G$ and $H$, will represent the graph $G\cup H=(V(G)\cup V(H),E(G)\cup E(H))$. The complement of $G$ will be denoted by $\overline{G}$. The $k_{th}$ power $G^k$ of a undirected graph $G$ will correspond to the graph $G$, with the addition of edges between any pair of vertices that have a path of length less or equal to $k$ between them. 
\\ \indent For $p, q$ and $r$ any integers, $p \mod q$ will denote the remainder of the euclidean division of $p$ by $q$, and $p\equiv r\mod q$ is equivalent to $p\mod q=r\mod q$.

\section{Introduction}
The graph packing problem has been widely studied in the literature, in many different ways. The general idea of the problem is  to find sufficient conditions on a set of graphs ($H_1$,$H_2$,...,$H_k$) or on a graph $G$, such that ($H_1$,$H_2$,...,$H_k$) will admit a packing in $G$, a packing being defined in the following way:

\begin{definition}
A packing of ($H_1$,$H_2$,...,$H_k$) in $G$ is a set of injections $\alpha_i: V(H_i) \rightarrow V(G), i=1,...,k$, such that $\alpha_i^*(E(H_i)) \cap \alpha_j^*(E(H_j)) = \emptyset$
for $i\neq j$  where $\alpha_i^*:E(H_i)\rightarrow E(G)$ is induced by $\alpha_i$.
\end{definition}

\bigbreak

If for all $i \in [1,k], |V(H_i)|=|V(G)|=n$, the $k$ injections are actually bijections. In the case of a $k$\textit{-placement}, where the vertices of the same graph $H$ are sent $k$ times to the vertices of $K_n$, we can send the images of the vertices of $H$ in the vertices of $H$ instead of the ones of $K_n$. Thus, the bijections are actually permutations on the vertices of $H$. In this context, we call a \textit{fixed point} any vertex $\alpha \in V(G)$ such that for all $i \in [1,k], \sigma_i(\alpha)=\alpha$.

\bigbreak 

The first results concerning the problem of graph packing focused on the packings into the complete graph $K_n$. In this case, the only conditions for the injections to form a packing is that the induced images of the edges should not intersect. Spencer and Sauer \cite{sauer1978edge} got first interested in the packing of two graphs, and found two sufficient conditions on the graphs for them to admit a packing in $K_n$.

\begin{theorem} \cite{sauer1978edge}  If $|E(H_1)|.|E(H_2)|\leq \dbinom{n}{2}$, then $(H_1,H_2)$ pack into $K_n$.
\end{theorem} 

\begin{theorem} \cite{sauer1978edge}   If $2\Delta(H_1) \Delta(H_2)< n$, then $(H_1,H_2)$ pack into $K_n$.
\end{theorem} 

\noindent Those two conditions are quite representative of two common approaches in graph packing theory, one that is interested in bounding the size of the graphs, and the other in bounding their maximal degree. Concerning the more general case of the packing of $k$ graphs, Bollobás and Eldridge formulated the following conjecture, still focusing on small sizes graphs:

\begin{conjecture} \cite{bollobas1978packings}
If $|E(H_1)|,|E(H_2)|,...,|E(H_k)|\leq n-k$, then $(H_1,H_2,...,H_k)$ pack in $K_n$.
\end{conjecture}

The conjecture has been proven for $k=2$ \cite{sauer1978edge} and $k=3$ \cite{kheddouci2001packing}. For $k\geq 4$, the conjecture remains open and is still the subject of many researches, as it is considered to be one of the most important open problems in graph packing theory. Another very important conjecture, about the packing of two graphs in $K_n$, is also interested on the maximal degree of two graphs:

\begin{conjecture} \cite{bollobas1978packings}
If $(\Delta(H_1)+1)(\Delta(H_2)+1)\leq n+1$, then $(H_1,H_2)$ pack into $K_n$.
\end{conjecture}

\bigbreak

Most of the studies concerning graph packing focused on the conditions on the graphs $H_i$ to pack into the complete graph $K_n$. However, some authors chose a different point of view, which is to fix the graphs $H_1$,$H_2$,...,$H_k$ and search for a minimal size graph $G$ that would make the packing of the $H_i$ in $G$ possible. A great part of this work has focused on the cases where the $H_i$ were trees.

\begin{theorem}\cite{kheddouci2000packing,kaneko2006packing}  For $n>3$, $T$ a tree of order $n$, $T\neq S_n$, there exists a 2-placement $\sigma$ of $T$ such that $\sigma(T) \subset T^3$.
\end{theorem} 

\bigbreak

If the reader wants to go into more details concerning the packing of graphs, he is invited to refer to the numerous surveys on the subject \cite{wozniak2004packing,yap1988packing}.
\bigbreak

In 2013, Duchene, Kheddouci, Nowakowski, and Tahraoui introduced the notion of labeled packings of graphs. They gave the following definition of a $p$-labeled packing of $k$ copies of a graph $G$:
\begin{definition} \cite{duchene2013labeled,tahraoui2012coloring}  For $p\geq 1$, $G$ a graph of order less or equal to $n$, and $f$ a mapping from $V(K_n)$ to a set of labels of cardinality $p$, $f$ is a $p$-labeled packing of $k$ copies of $G$ into $K_n$ if and only if for every $i$ in $\{1,2,...,k\}$, there exists an injection $\sigma_i: V(G)\rightarrow V(K_n)$ such that:
$
\\
\left\{
   \begin{array}{ll}
       \mbox{For every } i \neq j, \sigma_i^*(E(G))\cap \sigma_j^*(E(G))=\emptyset, \mbox{ where } \sigma_i^*: E(G)\rightarrow E(K_n) \mbox{ is induced by } \sigma_i\\
      \mbox{For every } v \in V(G), f(\sigma_1(v))=f(\sigma_2(v))=...=f(\sigma_k(v)).\\
   \end{array}
\right.
$
\end{definition}

\noindent Differently said, a $p$-labeled packing of $k$ copies of $G$ is a labeling of the vertices of $K_n$ with exactly $p$ distinct labels, such that there exists a set of $k$ injections from the vertices of $G$ to the vertices of $K_n$ such that the images of the edges of $G$ never overlap in $K_n$, and for any vertex $v$, the $k$ images of $v$ have the same label. It can be seen as a generalization of the problem of graph packing, as for $p=1$, all the vertices of $G$ have the same label and can thus be sent to any other vertex, as it is the case in the classic graph packing.
\\ \indent This definition allowed the authors to introduce a new parameter for a graph $G$, called the $k$-labeled packing number of $G$:

\begin{definition}\cite{duchene2013labeled,tahraoui2012coloring} The $k$-labeled packing number of $G$, denoted by $\lambda^k(G)$, is the largest $p$ such that $G$ admits a $p$-labeled packing of $k$ copies of $G$.
\end{definition}

The authors linked this new parameter to the cycle decomposition of permutations in the following observation:

\begin{observation}\cite{duchene2013labeled,tahraoui2012coloring} The notion of $k$-labeled packing number is closely related to the cycle decomposition of the packings as permutations.\end{observation} 

Indeed, for $k=2$, if $(id,\sigma_2)$ is a packing of $2$ copies of $G$ decomposed into $p$ disjoint cycles, then we know that in $\sigma_2$, the vertices of $G$ are sent to vertices that belong to the same cycle of the permutation, so that we can always associate one label to each of the $p$ cycles. That way, we can obtain a valid $p$-labeled packing of $2$ copies of $G$, so that we have $\lambda^2(G)\geq p$.
\\ \indent The authors, among other results, analyzed the value of $\lambda^k$ for cycles, and proved the following:

\begin{theorem}\cite{duchene2013labeled,tahraoui2012coloring}  Every cycle $C_n$ of order $n\leq 2k$ does not admit any $k$-placement. For every cycle $C_n$ of order $n=2km+x$, $k,m\geq 2$, $x<2k$, we have:
\\$\lambda^k(C_n) = \left\{
    \begin{array}{ll}
        \frac{n}{2}+1 & \mbox{ if } m=2, k>2 \mbox{ and } x\equiv 2 \mod 4\\
        \lfloor \frac{n}{2} \rfloor +m+1 & \mbox{ if } x=2k-1\\
        \lfloor \frac{n}{2} \rfloor +m & \mbox{ otherwise.}\\
    \end{array}
\right.
$
\end{theorem} 

\bigbreak

\noindent Concerning the left values of $n$, they proposed the following conjecture:

\begin{conjecture}\cite{duchene2013labeled,tahraoui2012coloring}  For every cycle $C_n$ of order $n=2k+x$, $k\geq 2$, $1\leq x \leq 2k-1$, we have:
\\$\lambda^k(C_n) = \left\{
    \begin{array}{ll}
        2 & \mbox{ if } x=1 \mbox{ and } k \mbox{ is even }\\
        x+2 & \mbox{ otherwise.}\\
    \end{array}
\right.
$
\label{THEconjecture}
\end{conjecture}

\bigbreak

\noindent To support this conjecture, they gave the following results, that focus on some particular values of $k$ and $x$:

\begin{theorem}\cite{duchene2013labeled,tahraoui2012coloring}
Let $C_n$ be the cycle of order $n=2k+x$, $k \geq 2$, $2k-3 \leq x \leq 2k-1$ and $(k,n) \neq (2,5)$. We have $\lambda^k(C_n) = x+2$.
\label{xbetween}
\end{theorem}

\begin{theorem}\cite{duchene2013labeled,tahraoui2012coloring}
Let $C_n$ be the cycle of order $n=2k+x$, $k$ a power of $2$, $x=1$. We have $\lambda^k(C_n) = 2$.
\label{x1_kpowerof2}
\end{theorem}

\begin{theorem}\cite{duchene2013labeled,tahraoui2012coloring}
Let $C_n$ be the cycle of order $n=2k+x$, $k$ prime, $x=1$. We have $\lambda^k(C_n) \leq x+2$.
\label{x1_kprime}
\end{theorem}

\begin{theorem}\cite{duchene2013labeled,tahraoui2012coloring}
Let $C_n$ be the cycle of order $n=2k+x$, $k$ even, $x=2$. We have $\lambda^k(C_n) \leq x+2$.
\label{x2_keven}
\end{theorem}

\section{Labeled packing of cycles}

In this part, we start by disproving Conjecture \ref{THEconjecture}. Indeed, we give a counter example to show that the upper bound it gives is not valid in general. However, we give some particular cases for which the upper bound is valid, and we prove that the lower bound always holds.

\bigbreak

The counter example uses the values $k=9$, $x=3$, and therefore $n=21$. We show that there exists a $7$-labeled packing of $9$ copies of $C_{21}$, so that $\lambda^9(C_{21}) \geq 7$, while Conjecture \ref{THEconjecture} would give $\lambda^9(C_{21}) = 3 +2 = 5$. We take the $7$-labeled packing of $9$ copies of $C_{21}$ presented in Figure \ref{counterexamplefigure}.

\bigbreak

\begin{figure}

	\centering
	\begin{tikzpicture}[scale=0.18]
    
    %Graphs
    
    \filldraw (10,0) circle (5pt);
    \draw (10,0) node[right]{$a$};
    \filldraw (9.6,2.9) circle (5pt);
    \draw (9.6,2.9) node[right]{$b$};
    \filldraw (8.3,5.6) circle (5pt);
    \draw (8.3,5.6) node[right]{$c$};
    \filldraw (6.2,7.8) circle (5pt);
    \draw (6.2,7.8) node[above]{$a$};
    \filldraw (3.7,9.3) circle (5pt);
    \draw (3.7,9.3) node[above]{$d$};
    \filldraw (0.7,10) circle (5pt);
    \draw (0.7,10) node[above]{$e$};
    \filldraw (-2.2,9.7) circle (5pt);
    \draw (-2.2,9.7) node[above]{$a$};
    \filldraw (-5,8.7) circle (5pt);
    \draw (-5,8.7) node[above]{$f$};
    \filldraw (-7.3,6.8) circle (5pt);
    \draw (-7.3,6.8) node[left]{$b$};
    \filldraw (-9,4.3) circle (5pt);
    \draw (-9,4.3) node[left]{$g$};
    \filldraw (-9.9,1.5) circle (5pt);
    \draw (-9.9,1.5) node[left]{$d$};
    \filldraw (-9.9,-1.5) circle (5pt);
    \draw (-9.9,-1.5) node[left]{$b$};
    \filldraw (-9,-4.3) circle (5pt);
    \draw (-9,-4.3) node[left]{$e$};
    \filldraw (-7.3,-6.8) circle (5pt);
    \draw (-7.3,-6.8) node[below]{$c$};
    \filldraw (-5,-8.7) circle (5pt);
    \draw (-5,-8.7) node[below]{$f$};
    \filldraw (-2.2,-9.7) circle (5pt);
    \draw (-2.2,-9.7) node[below]{$e$};
    \filldraw (0.7,-10) circle (5pt);
    \draw (0.7,-10) node[below]{$g$};
    \filldraw (3.7,-9.3) circle (5pt);
    \draw (3.7,-9.3) node[below]{$c$};
    \filldraw (6.2,-7.8) circle (5pt);
    \draw (6.2,-7.8) node[below]{$d$};
    \filldraw (8.3,-5.6) circle (5pt);
    \draw (8.3,-5.6) node[right]{$f$};
    \filldraw (9.6,-2.9) circle (5pt);
    \draw (9.6,-2.9) node[right]{$g$};
    
    \draw[thick] (10,0) -- (9.6,2.9) -- (8.3,5.6) -- (6.2,7.8) -- (3.7,9.3) -- (0.7,10) -- (-2.2,9.7) -- (-5,8.7) -- (-7.3,6.8) -- (-9,4.3) -- (-9.9,1.5) -- (-9.9,-1.5) -- (-9,-4.3) -- (-7.3,-6.8) -- (-5,-8.7) -- (-2.2,-9.7) -- (0.7,-10) -- (3.7,-9.3) -- (6.2,-7.8) -- (8.3,-5.6) -- (9.6,-2.9) -- (10,0);

    \filldraw (40,0) circle (5pt);
    \draw (40,0) node[right]{$a$};
    \filldraw (39.6,2.9) circle (5pt);
    \draw (39.6,2.9) node[right]{$b$};
    \filldraw (38.3,5.6) circle (5pt);
    \draw (38.3,5.6) node[right]{$c$};
    \filldraw (36.2,7.8) circle (5pt);
    \draw (36.2,7.8) node[above]{$a$};
    \filldraw (33.7,9.3) circle (5pt);
    \draw (33.7,9.3) node[above]{$d$};
    \filldraw (30.7,10) circle (5pt);
    \draw (30.7,10) node[above]{$e$};
    \filldraw (27.8,9.7) circle (5pt);
    \draw (27.8,9.7) node[above]{$a$};
    \filldraw (25,8.7) circle (5pt);
    \draw (25,8.7) node[above]{$f$};
    \filldraw (22.7,6.8) circle (5pt);
    \draw (22.7,6.8) node[left]{$b$};
    \filldraw (21,4.3) circle (5pt);
    \draw (21,4.3) node[left]{$g$};
    \filldraw (20.1,1.5) circle (5pt);
    \draw (20.1,1.5) node[left]{$d$};
    \filldraw (20.1,-1.5) circle (5pt);
    \draw (20.1,-1.5) node[left]{$b$};
    \filldraw (21,-4.3) circle (5pt);
    \draw (21,-4.3) node[left]{$e$};
    \filldraw (22.7,-6.8) circle (5pt);
    \draw (22.7,-6.8) node[below]{$c$};
    \filldraw (25,-8.7) circle (5pt);
    \draw (25,-8.7) node[below]{$f$};
    \filldraw (27.8,-9.7) circle (5pt);
    \draw (27.8,-9.7) node[below]{$e$};
    \filldraw (30.7,-10) circle (5pt);
    \draw (30.7,-10) node[below]{$g$};
    \filldraw (33.7,-9.3) circle (5pt);
    \draw (33.7,-9.3) node[below]{$c$};
    \filldraw (36.2,-7.8) circle (5pt);
    \draw (36.2,-7.8) node[below]{$d$};
    \filldraw (38.3,-5.6) circle (5pt);
    \draw (38.3,-5.6) node[right]{$f$};
    \filldraw (39.6,-2.9) circle (5pt);
    \draw (39.6,-2.9) node[right]{$g$};
    
    \draw[thick] (40,0) -- (22.7,6.8) -- (22.7,-6.8) -- (36.2,7.8) -- (20.1,1.5) --(21,-4.3) -- (27.8,9.7) -- (27.8,9.7) -- (25,-8.7) -- (20.1,-1.5) -- (39.6,-2.9) -- (36.2,-7.8) -- (39.6,2.9) -- (27.8,-9.7) -- (38.3,5.6) -- (38.3,-5.6) -- (30.7,10) -- (21,4.3) -- (33.7,-9.3) -- (33.7,9.3) -- (25,8.7) -- (30.7,-10) -- (40,0);

    ['a1', 'b2', 'c2', 'a2', 'd2', 'e2', 'a3', 'f2', 'b3', 'g1', 'd3', 'b1', 'e3', 'c1', 'f3', 'e1', 'g2', 'c3', 'd1', 'f1', 'g3']

    \filldraw (70,0) circle (5pt);
    \draw (70,0) node[right]{$a$};
    \filldraw (69.6,2.9) circle (5pt);
    \draw (69.6,2.9) node[right]{$b$};
    \filldraw (68.3,5.6) circle (5pt);
    \draw (68.3,5.6) node[right]{$c$};
    \filldraw (66.2,7.8) circle (5pt);
    \draw (66.2,7.8) node[above]{$a$};
    \filldraw (63.7,9.3) circle (5pt);
    \draw (63.7,9.3) node[above]{$d$};
    \filldraw (60.7,10) circle (5pt);
    \draw (60.7,10) node[above]{$e$};
    \filldraw (57.8,9.7) circle (5pt);
    \draw (57.8,9.7) node[above]{$a$};
    \filldraw (55,8.7) circle (5pt);
    \draw (55,8.7) node[above]{$f$};
    \filldraw (52.7,6.8) circle (5pt);
    \draw (52.7,6.8) node[left]{$b$};
    \filldraw (51,4.3) circle (5pt);
    \draw (51,4.3) node[left]{$g$};
    \filldraw (50.1,1.5) circle (5pt);
    \draw (50.1,1.5) node[left]{$d$};
    \filldraw (50.1,-1.5) circle (5pt);
    \draw (50.1,-1.5) node[left]{$b$};
    \filldraw (51,-4.3) circle (5pt);
    \draw (51,-4.3) node[left]{$e$};
    \filldraw (52.7,-6.8) circle (5pt);
    \draw (52.7,-6.8) node[below]{$c$};
    \filldraw (55,-8.7) circle (5pt);
    \draw (55,-8.7) node[below]{$f$};
    \filldraw (57.8,-9.7) circle (5pt);
    \draw (57.8,-9.7) node[below]{$e$};
    \filldraw (60.7,-10) circle (5pt);
    \draw (60.7,-10) node[below]{$g$};
    \filldraw (63.7,-9.3) circle (5pt);
    \draw (63.7,-9.3) node[below]{$c$};
    \filldraw (66.2,-7.8) circle (5pt);
    \draw (66.2,-7.8) node[below]{$d$};
    \filldraw (68.3,-5.6) circle (5pt);
    \draw (68.3,-5.6) node[right]{$f$};
    \filldraw (69.6,-2.9) circle (5pt);
    \draw (69.6,-2.9) node[right]{$g$};
    
    \draw[thick] (70,0) -- (50.1,-1.5) -- (68.3,5.6) -- (57.8,9.7) -- (66.2,-7.8) -- (57.8,-9.7) -- (66.2,7.8) -- (68.3,-5.6) --  (52.7,6.8) -- (69.6,-2.9) -- (50.1,1.5) -- (69.6,2.9) -- (51,-4.3) -- (63.7,-9.3) -- (55,8.7) -- (60.7,10) -- (60.7,-10) --  (52.7,-6.8) -- (63.7,9.3) -- (55,-8.7) -- (55,-8.7) -- (51,4.3) -- (70,0);
    
    ['a1', 'b3', 'c1', 'a3', 'd3', 'e3', 'a2', 'f3', 'b2', 'g1', 'd2', 'b1', 'e2', 'c3', 'f1', 'e1', 'g3', 'c2', 'd1', 'f2', 'g2']
    
    \filldraw (10,-30) circle (5pt);
    \draw (10,-30) node[right]{$a$};
    \filldraw (9.6,-27.1) circle (5pt);
    \draw (9.6,-27.1) node[right]{$b$};
    \filldraw (8.3,-24.6) circle (5pt);
    \draw (8.3,-24.6) node[right]{$c$};
    \filldraw (6.2,-22.2) circle (5pt);
    \draw (6.2,-22.2) node[above]{$a$};
    \filldraw (3.7,-20.7) circle (5pt);
    \draw (3.7,-20.7) node[above]{$d$};
    \filldraw (0.7,-20) circle (5pt);
    \draw (0.7,-20) node[above]{$e$};
    \filldraw (-2.2,-20.3) circle (5pt);
    \draw (-2.2,-20.3) node[above]{$a$};
    \filldraw (-5,-21.3) circle (5pt);
    \draw (-5,-21.3) node[above]{$f$};
    \filldraw (-7.3,-23.8) circle (5pt);
    \draw (-7.3,-23.8) node[left]{$b$};
    \filldraw (-9,-25.7) circle (5pt);
    \draw (-9,-25.7) node[left]{$g$};
    \filldraw (-9.9,-28.5) circle (5pt);
    \draw (-9.9,-28.5) node[left]{$d$};
    \filldraw (-9.9,-31.5) circle (5pt);
    \draw (-9.9,-31.5) node[left]{$b$};
    \filldraw (-9,-34.3) circle (5pt);
    \draw (-9,-34.3) node[left]{$e$};
    \filldraw (-7.3,-36.8) circle (5pt);
    \draw (-7.3,-36.8) node[below]{$c$};
    \filldraw (-5,-38.7) circle (5pt);
    \draw (-5,-38.7) node[below]{$f$};
    \filldraw (-2.2,-39.7) circle (5pt);
    \draw (-2.2,-39.7) node[below]{$e$};
    \filldraw (0.7,-40) circle (5pt);
    \draw (0.7,-40) node[below]{$g$};
    \filldraw (3.7,-39.3) circle (5pt);
    \draw (3.7,-39.3) node[below]{$c$};
    \filldraw (6.2,-37.8) circle (5pt);
    \draw (6.2,-37.8) node[below]{$d$};
    \filldraw (8.3,-35.6) circle (5pt);
    \draw (8.3,-35.6) node[right]{$f$};
    \filldraw (9.6,-32.9) circle (5pt);
    \draw (9.6,-32.9) node[right]{$g$};
    
    \draw[thick] (6.2,-22.2) -- (9.6,-27.1) -- (3.7,-39.3) -- (-2.2,-20.3) -- (3.7,-20.7) -- (-2.2,-39.7) -- (10,-30) -- (-5,-21.3) -- (-9.9,-31.5) -- (-9,-25.7) -- (6.2,-37.8) -- (-7.3,-23.8) -- (-9,-34.3) -- (8.3,-24.6) -- (-5,-38.7) -- (0.7,-20) -- (9.6,-32.9) -- (-7.3,-36.8) -- (-9.9,-28.5) -- (8.3,-35.6) -- (0.7,-40) -- (6.2,-22.2);

    ['a2', 'b1', 'c3', 'a3', 'd1', 'e3', 'a1', 'f1', 'b3', 'g2', 'd3', 'b2', 'e2', 'c1', 'f2', 'e1', 'g1', 'c2', 'd2', 'f3', 'g3']

    \filldraw (40,-30) circle (5pt);
    \draw (40,-30) node[right]{$a$};
    \filldraw (39.6,-27.1) circle (5pt);
    \draw (39.6,-27.1) node[right]{$b$};
    \filldraw (38.3,-24.4) circle (5pt);
    \draw (38.3,-24.4) node[right]{$c$};
    \filldraw (36.2,-22.2) circle (5pt);
    \draw (36.2,-22.2) node[above]{$a$};
    \filldraw (33.7,-20.7) circle (5pt);
    \draw (33.7,-20.7) node[above]{$d$};
    \filldraw (30.7,-20) circle (5pt);
    \draw (30.7,-20) node[above]{$e$};
    \filldraw (27.8,-20.3) circle (5pt);
    \draw (27.8,-20.3) node[above]{$a$};
    \filldraw (25,-21.3) circle (5pt);
    \draw (25,-21.3) node[above]{$f$};
    \filldraw (22.7,-23.2) circle (5pt);
    \draw (22.7,-23.2) node[left]{$b$};
    \filldraw (21,-25.7) circle (5pt);
    \draw (21,-25.7) node[left]{$g$};
    \filldraw (20.1,-28.5) circle (5pt);
    \draw (20.1,-28.5) node[left]{$d$};
    \filldraw (20.1,-31.5) circle (5pt);
    \draw (20.1,-31.5) node[left]{$b$};
    \filldraw (21,-34.3) circle (5pt);
    \draw (21,-34.3) node[left]{$e$};
    \filldraw (22.7,-36.8) circle (5pt);
    \draw (22.7,-36.8) node[below]{$c$};
    \filldraw (25,-38.7) circle (5pt);
    \draw (25,-38.7) node[below]{$f$};
    \filldraw (27.8,-39.7) circle (5pt);
    \draw (27.8,-39.7) node[below]{$e$};
    \filldraw (30.7,-40) circle (5pt);
    \draw (30.7,-40) node[below]{$g$};
    \filldraw (33.7,-39.3) circle (5pt);
    \draw (33.7,-39.3) node[below]{$c$};
    \filldraw (36.2,-37.8) circle (5pt);
    \draw (36.2,-37.8) node[below]{$d$};
    \filldraw (38.3,-35.6) circle (5pt);
    \draw (38.3,-35.6) node[right]{$f$};
    \filldraw (39.6,-32.9) circle (5pt);
    \draw (39.6,-32.9) node[right]{$g$};
    
    \draw[thick] (36.2,-22.2) -- (22.7,-23.2) -- (33.7,-39.3) -- (40,-30) -- (20.1,-28.5) -- (27.8,-39.7) -- (27.8,-20.3) -- (38.3,-35.6) -- (20.1,-31.5) -- (30.7,-40) -- (33.7,-20.7) -- (39.6,-27.1) -- (30.7,-20) -- (38.3,-24.4) -- (25,-21.3) -- (21,-34.3) -- (21,-25.7) -- (22.7,-36.8) -- (36.2,-37.8) -- (25,-38.7) -- (39.6,-32.9) -- (36.2,-22.2);
    
    ['a2', 'b2', 'c3', 'a1', 'd2', 'e3', 'a3', 'f3', 'b3', 'g3', 'd1', 'b1', 'e1', 'c1', 'f1', 'e2', 'g2', 'c2', 'd3', 'f2', 'g1']
    
    \filldraw (70,-30) circle (5pt);
    \draw (70,-30) node[right]{$a$};
    \filldraw (69.6,-27.1) circle (5pt);
    \draw (69.6,-27.1) node[right]{$b$};
    \filldraw (68.3,-24.4) circle (5pt);
    \draw (68.3,-24.4) node[right]{$c$};
    \filldraw (66.2,-22.2) circle (5pt);
    \draw (66.2,-22.2) node[above]{$a$};
    \filldraw (63.7,-20.7) circle (5pt);
    \draw (63.7,-20.7) node[above]{$d$};
    \filldraw (60.7,-20) circle (5pt);
    \draw (60.7,-20) node[above]{$e$};
    \filldraw (57.8,-20.3) circle (5pt);
    \draw (57.8,-20.3) node[above]{$a$};
    \filldraw (55,-21.3) circle (5pt);
    \draw (55,-21.3) node[above]{$f$};
    \filldraw (52.7,-23.2) circle (5pt);
    \draw (52.7,-23.2) node[left]{$b$};
    \filldraw (51,-25.7) circle (5pt);
    \draw (51,-25.7) node[left]{$g$};
    \filldraw (50.1,-28.5) circle (5pt);
    \draw (50.1,-28.5) node[left]{$d$};
    \filldraw (50.1,-31.5) circle (5pt);
    \draw (50.1,-31.5) node[left]{$b$};
    \filldraw (51,-34.3) circle (5pt);
    \draw (51,-34.3) node[left]{$e$};
    \filldraw (52.7,-36.8) circle (5pt);
    \draw (52.7,-36.8) node[below]{$c$};
    \filldraw (55,-38.7) circle (5pt);
    \draw (55,-38.7) node[below]{$f$};
    \filldraw (57.8,-39.7) circle (5pt);
    \draw (57.8,-39.7) node[below]{$e$};
    \filldraw (60.7,-40) circle (5pt);
    \draw (60.7,-40) node[below]{$g$};
    \filldraw (63.7,-39.3) circle (5pt);
    \draw (63.7,-39.3) node[below]{$c$};
    \filldraw (66.2,-37.8) circle (5pt);
    \draw (66.2,-37.8) node[below]{$d$};
    \filldraw (68.3,-35.6) circle (5pt);
    \draw (68.3,-35.6) node[right]{$f$};
    \filldraw (69.6,-32.9) circle (5pt);
    \draw (69.6,-32.9) node[right]{$g$};

    \draw[thick] (66.2,-22.2) -- (50.1,-31.5) -- (52.7,-36.8) -- (57.8,-20.3) -- (50.1,-28.5) -- (60.7,-20) -- (70,-30) -- (68.3,-35.6) -- (69.6,-27.1) -- (69.6,-32.9) -- (63.7,-20.7) -- (52.7,-23.2) -- (57.8,-39.7) -- (63.7,-39.3) -- (55,-38.7) -- (51,-34.3) -- (60.7,-40)  -- (68.3,-24.4) -- (66.2,-37.8) -- (55,-21.3) -- (51,-25.7) -- (66.2,-22.2);
    
    ['a2', 'b3', 'c2', 'a3', 'd2', 'e1', 'a1', 'f3', 'b1', 'g1', 'd1', 'b2', 'e3', 'c3', 'f2', 'e2', 'g3', 'c1', 'd3', 'f1', 'g2']

      \filldraw (10,-60) circle (5pt);
    \draw (10,-60) node[right]{$a$};
    \filldraw (9.6,-57.1) circle (5pt);
    \draw (9.6,-57.1) node[right]{$b$};
    \filldraw (8.3,-54.6) circle (5pt);
    \draw (8.3,-54.6) node[right]{$c$};
    \filldraw (6.2,-52.2) circle (5pt);
    \draw (6.2,-52.2) node[above]{$a$};
    \filldraw (3.7,-50.7) circle (5pt);
    \draw (3.7,-50.7) node[above]{$d$};
    \filldraw (0.7,-50) circle (5pt);
    \draw (0.7,-50) node[above]{$e$};
    \filldraw (-2.2,-50.3) circle (5pt);
    \draw (-2.2,-50.3) node[above]{$a$};
    \filldraw (-5,-51.3) circle (5pt);
    \draw (-5,-51.3) node[above]{$f$};
    \filldraw (-7.3,-53.8) circle (5pt);
    \draw (-7.3,-53.8) node[left]{$b$};
    \filldraw (-9,-55.7) circle (5pt);
    \draw (-9,-55.7) node[left]{$g$};
    \filldraw (-9.9,-58.5) circle (5pt);
    \draw (-9.9,-58.5) node[left]{$d$};
    \filldraw (-9.9,-61.5) circle (5pt);
    \draw (-9.9,-61.5) node[left]{$b$};
    \filldraw (-9,-64.3) circle (5pt);
    \draw (-9,-64.3) node[left]{$e$};
    \filldraw (-7.3,-66.8) circle (5pt);
    \draw (-7.3,-66.8) node[below]{$c$};
    \filldraw (-5,-68.7) circle (5pt);
    \draw (-5,-68.7) node[below]{$f$};
    \filldraw (-2.2,-69.7) circle (5pt);
    \draw (-2.2,-69.7) node[below]{$e$};
    \filldraw (0.7,-70) circle (5pt);
    \draw (0.7,-70) node[below]{$g$};
    \filldraw (3.7,-69.3) circle (5pt);
    \draw (3.7,-69.3) node[below]{$c$};
    \filldraw (6.2,-67.8) circle (5pt);
    \draw (6.2,-67.8) node[below]{$d$};
    \filldraw (8.3,-65.6) circle (5pt);
    \draw (8.3,-65.6) node[right]{$f$};
    \filldraw (9.6,-62.9) circle (5pt);
    \draw (9.6,-62.9) node[right]{$g$};
    
    \draw[thick] (-2.2,-50.3) -- (9.6,-57.1) -- (-7.3,-66.8) -- (10,-60) -- (3.7,-50.7) -- (-9,-64.3) -- (6.2,-52.2) -- (-5,-68.7) -- (-7.3,-53.8) -- (0.7,-70) -- (6.2,-67.8) -- (-9.9,-61.5) -- (0.7,-50) -- (3.7,-69.3) -- (8.3,-65.6) -- (-2.2,-69.7) -- (-9,-55.7) -- (8.3,-54.6) -- (-9.9,-58.5) -- (-5,-51.3) -- (9.6,-62.9) -- (-2.2,-50.3);
    
    ['a3', 'b1', 'c2', 'a1', 'd1', 'e2', 'a2', 'f2', 'b2', 'g3', 'd3', 'b3', 'e1', 'c3', 'f3', 'e3', 'g2', 'c1', 'd2', 'f1', 'g1']

    \filldraw (40,-60) circle (5pt);
    \draw (40,-60) node[right]{$a$};
    \filldraw (39.6,-57.1) circle (5pt);
    \draw (39.6,-57.1) node[right]{$b$};
    \filldraw (38.3,-54.4) circle (5pt);
    \draw (38.3,-54.4) node[right]{$c$};
    \filldraw (36.2,-52.2) circle (5pt);
    \draw (36.2,-52.2) node[above]{$a$};
    \filldraw (33.7,-50.7) circle (5pt);
    \draw (33.7,-50.7) node[above]{$d$};
    \filldraw (30.7,-50) circle (5pt);
    \draw (30.7,-50) node[above]{$e$};
    \filldraw (27.8,-50.3) circle (5pt);
    \draw (27.8,-50.3) node[above]{$a$};
    \filldraw (25,-51.3) circle (5pt);
    \draw (25,-51.3) node[above]{$f$};
    \filldraw (22.7,-53.2) circle (5pt);
    \draw (22.7,-53.2) node[left]{$b$};
    \filldraw (21,-55.7) circle (5pt);
    \draw (21,-55.7) node[left]{$g$};
    \filldraw (20.1,-58.5) circle (5pt);
    \draw (20.1,-58.5) node[left]{$d$};
    \filldraw (20.1,-61.5) circle (5pt);
    \draw (20.1,-61.5) node[left]{$b$};
    \filldraw (21,-64.3) circle (5pt);
    \draw (21,-64.3) node[left]{$e$};
    \filldraw (22.7,-66.8) circle (5pt);
    \draw (22.7,-66.8) node[below]{$c$};
    \filldraw (25,-68.7) circle (5pt);
    \draw (25,-68.7) node[below]{$f$};
    \filldraw (27.8,-69.7) circle (5pt);
    \draw (27.8,-69.7) node[below]{$e$};
    \filldraw (30.7,-70) circle (5pt);
    \draw (30.7,-70) node[below]{$g$};
    \filldraw (33.7,-69.3) circle (5pt);
    \draw (33.7,-69.3) node[below]{$c$};
    \filldraw (36.2,-67.8) circle (5pt);
    \draw (36.2,-67.8) node[below]{$d$};
    \filldraw (38.3,-65.6) circle (5pt);
    \draw (38.3,-65.6) node[right]{$f$};
    \filldraw (39.6,-62.9) circle (5pt);
    \draw (39.6,-62.9) node[right]{$g$};
    
    \draw[thick] (27.8,-50.3) -- (22.7,-53.2) -- (38.3,-54.4) -- (40,-60) -- (36.2,-67.8) -- (30.7,-50) -- (36.2,-52.2) -- (25,-51.3) -- (39.6,-57.1) -- (21,-55.7) -- (33.7,-50.7) -- (20.1,-61.5) -- (27.8,-69.7) -- (22.7,-66.8) -- (38.3,-65.6) -- (21,-64.3) -- (39.6,-62.9) -- (33.7,-69.3) -- (20.1,-58.5) -- (25,-68.7) -- (30.7,-70) -- (27.8,-50.3);

   ['a3', 'b2', 'c1', 'a1', 'd3', 'e1', 'a2', 'f1', 'b1', 'g2', 'd1', 'b3', 'e3', 'c2', 'f3', 'e2', 'g1', 'c3', 'd2', 'f2', 'g3']

        \filldraw (70,-60) circle (5pt);
    \draw (70,-60) node[right]{$a$};
    \filldraw (69.6,-57.1) circle (5pt);
    \draw (69.6,-57.1) node[right]{$b$};
    \filldraw (68.3,-54.4) circle (5pt);
    \draw (68.3,-54.4) node[right]{$c$};
    \filldraw (66.2,-52.2) circle (5pt);
    \draw (66.2,-52.2) node[above]{$a$};
    \filldraw (63.7,-50.7) circle (5pt);
    \draw (63.7,-50.7) node[above]{$d$};
    \filldraw (60.7,-50) circle (5pt);
    \draw (60.7,-50) node[above]{$e$};
    \filldraw (57.8,-50.3) circle (5pt);
    \draw (57.8,-50.3) node[above]{$a$};
    \filldraw (55,-51.3) circle (5pt);
    \draw (55,-51.3) node[above]{$f$};
    \filldraw (52.7,-53.2) circle (5pt);
    \draw (52.7,-53.2) node[left]{$b$};
    \filldraw (51,-55.7) circle (5pt);
    \draw (51,-55.7) node[left]{$g$};
    \filldraw (50.1,-58.5) circle (5pt);
    \draw (50.1,-58.5) node[left]{$d$};
    \filldraw (50.1,-61.5) circle (5pt);
    \draw (50.1,-61.5) node[left]{$b$};
    \filldraw (51,-64.3) circle (5pt);
    \draw (51,-64.3) node[left]{$e$};
    \filldraw (52.7,-66.8) circle (5pt);
    \draw (52.7,-66.8) node[below]{$c$};
    \filldraw (55,-68.7) circle (5pt);
    \draw (55,-68.7) node[below]{$f$};
    \filldraw (57.8,-69.7) circle (5pt);
    \draw (57.8,-69.7) node[below]{$e$};
    \filldraw (60.7,-70) circle (5pt);
    \draw (60.7,-70) node[below]{$g$};
    \filldraw (63.7,-69.3) circle (5pt);
    \draw (63.7,-69.3) node[below]{$c$};
    \filldraw (66.2,-67.8) circle (5pt);
    \draw (66.2,-67.8) node[below]{$d$};
    \filldraw (68.3,-65.6) circle (5pt);
    \draw (68.3,-65.6) node[right]{$f$};
    \filldraw (69.6,-62.9) circle (5pt);
    \draw (69.6,-62.9) node[right]{$g$};

    \draw[thick] (57.8,-50.3) -- (50.1,-61.5) -- (63.7,-69.3) -- (66.2,-52.2) -- (66.2,-67.8) -- (51,-64.3) -- (70,-60) -- (55,-68.7) -- (69.6,-57.1) -- (60.7,-70) -- (50.1,-58.5) -- (52.7,-53.2) -- (60.7,-50) -- (52.7,-66.8) -- (55,-51.3) -- (57.8,-69.7) -- (69.6,-62.9) -- (68.3,-54.4) -- (63.7,-50.7) -- (68.3,-65.6) -- (51,-55.7) -- (57.8,-50.3);

['a3', 'b3', 'c3', 'a2', 'd3', 'e2', 'a1', 'f2', 'b1', 'g3', 'd2', 'b2', 'e1', 'c2', 'f1', 'e3', 'g1', 'c1', 'd1', 'f3', 'g2']

	\end{tikzpicture}
    \caption{A $7$-labeled packing of $k=9$ copies of $C_{21}$. The labels are represented by the letters on the vertices.}
	\label{counterexamplefigure}
\end{figure}
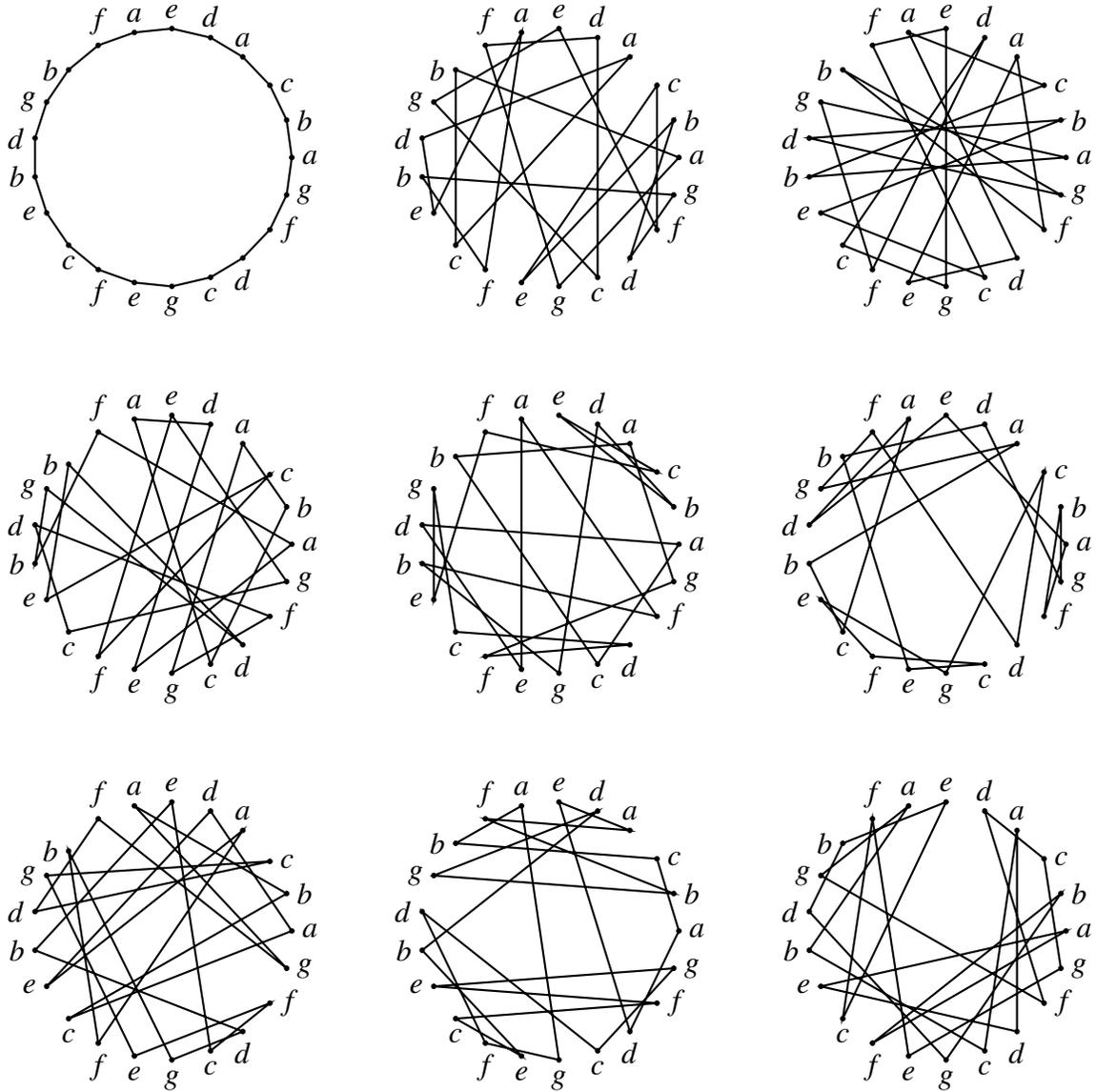
Despite the fact that Conjecture \ref{THEconjecture} is now proven to be false, it gives a good lower bound for $\lambda^k(C_n)$, as expressed in the following theorem:

\begin{theorem}  For every cycle $C_n$ of order $n=2k+x$, $k\geq 2$, $1\leq x \leq 2k-1$, we have:
\\$\lambda^k(C_n) \geq \left\{
    \begin{array}{ll}
        2 & \mbox{ if } x=1 \mbox{ and } k \mbox{ is even }\\
        x+2 & \mbox{ otherwise.}\\
    \end{array}
\right.
$
\label{borne_sup_theorem}
\end{theorem}

\begin{proof} We split the proof into five cases, depending on the values of $k$ and $x$, as follows:

\begin{itemize}
\item[1.] $k$ is odd
\begin{itemize}
\item[a)] $1\leq x\leq k-1$
\item[b)] $k \leq x \leq 2k-1$
\end{itemize}
\item[2.] $k$ is even and $x=1$
\item[3.] $k$ is even and $x>1$
\begin{itemize}
\item[a)] $2\leq x\leq k-1$
\item[b)] $k\leq x \leq 2k-1$
\end{itemize}
\end{itemize}

For each case, we construct a labeling $f$ with the given number of labels, and give the associated packing $\sigma=\{\sigma_j, 1\leq j\leq k\}$.

\bigbreak

Let $V(C_n)=\{v_0,...,v_n\}$ be the vertices of $C_n$. For each of those cases, we will partition $V(C_n)$ into a set $F$ containing all the fixed points, and a set $V_l$ for each other label $l$, containing all the vertices having this label:
\\$F=\{v\in V(C_n) : \mbox{ for all } j\in [1,k], \sigma_j(v)=v\}$
\\$V_l=\{v\in V(C_n) : f(v)=l\}$ for any l.
\\We will also partition the set of edges $E(C_n)$ depending on the labels of their vertices:
\\$E_F=\{\{u,v\}\in E(C_n) :  u\in F \mbox{ or } v\in F\}$
\\$E_{p,q}=\{\{u,v\}\in E(C_n) :  u\in V_p \mbox{ and } v\in V_q\}$.

\bigbreak

For $f$ to be a labeled packing with respect to $\sigma$, we must have, for all $j \neq j'$ and $v \in V(C_n)$, $f(\sigma_j(v))=f(\sigma_{j'}(v))$. In any case, for all $v\in F$, since $\sigma_j(v)=v$, we know that $f(\sigma_j(v))=f(v)=f(\sigma_{j'}(v))$. Thus, we only have to show that for all $v\not \in F$ and $j\neq j'$, $f(\sigma_j(v))=f(\sigma_{j'}(v))$.

\bigbreak

For $\sigma$ to be a packing, we must have, for all $j \neq j', \sigma_j^*(E(C_n))\cap \sigma_{j'}^*(E(C_n))=\emptyset$. Assuming that $f$ is a labeled packing with respect to $\sigma$, we already know that for all labels $p,q$, if $X\neq E_{p,q}$, then $\sigma_j^*(E_{p,q})\cap \sigma_{j'}^*(X)=\emptyset$.
\\ \indent We will show for each case that for all $p,q$ and $\{u,v\},\{u',v'\}\in E_{p,q}$ , $\sigma_j(\{u,v\})=\sigma_{j'}(\{u',v'\})$ implies $j=j'$, so that for all $j\neq j', \sigma_j^*(E_{p,q})\cap \sigma_{j'}^*(E_{p,q})=\emptyset$. 
\\ \indent To show that $\sigma_j^*(E_F)\cap \sigma_{j'}^*(E_F)=\emptyset$, since each edge of $E_F$ contains one fixed point, we only have to show that its two neighbors $u$ and $v$ verify $\sigma_j(u)\neq \sigma_{j'}(u)$, $\sigma_j(v)\neq \sigma_{j'}(v)$, $\sigma_j(u)\neq \sigma_{j'}(v)$ and $\sigma_j(v)\neq \sigma_{j'}(u)$. For the two first conditions, we will always choose $\sigma$ such that for all $v\not \in F$ and $j\neq j'$, $\sigma_j(v) \neq \sigma_{j'}(v)$. The two last conditions are respected if $f(u)\neq f(v)$, assuming that $f$ is a labeled packing with respect to $\sigma$. If $f(u)=f(v)$, we will show them.

\bigbreak

For each of the five cases, we are thus going to give a construction and prove that:

\begin{itemize}
\item[*] For each $v\not \in F$ and $j\neq j'$, $f(\sigma_j(v))=f(\sigma_{j'}(v))$.
\item[*] For any $p,q$ and any $\{u,v\},\{u',v'\}\in E_{p,q}$ , if $\sigma_j(\{u,v\})=\sigma_{j'}(\{u',v'\})$, then $j=j'$.
\item[*] For any $v\not \in F$ and $j\neq j'$, $\sigma_j(v) \neq \sigma_{j'}(v)$.
\item[*] For each $u,v$ such that there exists $x\in F$ such that $\{u,x\}, \{x,v\} \in E(C_n)$, if $f(u)=f(v)$, then, for all $j\neq j'$, $\sigma_j(u)\neq \sigma_{j'}(v)$ and $\sigma_j(v)\neq \sigma_{j'}(u)$.
\end{itemize}

\begin{itemize}[leftmargin=*]

\item[1.] $k$ is odd:
\begin{itemize}[leftmargin=*]
\item[a)] $1\leq x\leq k-1$:

\bigbreak

In this case, the principle of the construction is the following: For $x=1$, we label with $1$ the vertex $v_{2k}$, with $a$ every other vertex of even number, and with $b$ every other vertex of odd number. Then, we create the first copy of $C_n$ by joining the vertices in the following order: $$v_0,v_{n-2},v_1,v_{n-3},v_2,v_{n-4},...,v_{\frac{n-1}{2}-1},v_{\frac{n-1}{2}},v_{2k},v_0.$$ We create each other copy by following the same order as the previous one but with every number but $2k$ added to $2$ modulo $n-1$. Note that, since $k$ is odd, the obtained packing is exactly the decomposition of $K_n$ given by Walecki \cite{alspach2008wonderful}. For $x>1$, we add fixed points to the previous construction, on $x-1$ of the $k$ edges that link a vertex $u$ of label $a$ to a vertex $v$ of label $b$, cutting every edge into two edges.

\bigbreak

To express it in a more formal way, we take the following labeling $f$, packing $\sigma=\{\sigma_j, 1\leq j\leq k\}$, and vertices and edges partitions of $C_n$:

\bigbreak

$f(v_i)=\left\{
    \begin{array}{ll}
        a & \mbox{ if } 0\leq i \leq 2k-1 \mbox{ and } i \mbox{ is even } \\
        b & \mbox{ if } 0\leq i \leq 2k-1 \mbox{ and } i \mbox{ is odd } \\
        i-2k+1 & \mbox{ if } i \geq 2k.\\
    \end{array}
\right.
$

\bigbreak

$\sigma_j(v_i)=
\left\{
    \begin{array}{ll}
        v_{i+2(j-1)\mod 2k} & \mbox{ if } i < 2k\\
        v_{i} & \mbox{ if } i \geq 2k.\\
    \end{array}
\right.$

\bigbreak

$F=\{v_i,2k\leq i\leq n-1\}$
\\$V_a=\{v_i,0\leq i \leq 2k-1 \mbox{ and } i\equiv 0 \mod 2\}$
\\$V_b=\{v_i,0\leq i \leq 2k-1 \mbox{ and } i\equiv 1 \mod 2\}$.

\bigbreak

$E_F=\{\{v_{i-1},v_{2k+i}\}\cup \{v_{2k-i},v_{2k+i}\}, 1\leq i\leq x-1\}\cup \{\{v_0,v_{2k}\},\{v_k,v_{2k}\}\}$
\\$E_{a,a}=\{\{v_i,v_{2k-i}\}, 1\leq i\leq k-1 \mbox{ and } i\equiv 0 \mod 2\}$
\\$E_{b,b}=\{\{v_i,v_{2k-i}\}, 1\leq i\leq k-1 \mbox{ and } i\equiv 1 \mod 2\}$
\\$E_{a,b}=\{\{v_i,v_{2k-1-i}\}, x-1\leq i\leq k-1\}$.

\bigbreak

The obtained graph $G$ is, as wanted, isomorphic to $C_n$. \\Indeed, for $x=1$, for $H$ a graph, if $V(H)=\{v_1,v_2,...,v_n\}$ and \\$E(H)=\{\{v_i,v_{i+1 \mod n}\},1\leq i\leq n\}$, we obviously have $H\simeq C_n$, and $G=\tau(H)$, with $\tau$ the permutation defined as follows:

$ \tau(v_i)=
\left\{
    \begin{array}{ll}
        v_{\frac{i}{2}} & \mbox{ if } i<2k \mbox{ and } i \equiv 0\mod 2\\
        v_{2k-\frac{i+1}{2}} & \mbox{ if } i<2k \mbox{ and } i \equiv 1\mod 2\\
        v_{i} & \mbox{ if } i=2k.\\
    \end{array}
\right.$

For $x>1$, the extra fixed point are just inserted in some of the edges of $C_{2k+1}$, so that we always keep $G \simeq C_n$.

\bigbreak

We now show that $f$ and $\sigma$ define a $(x+2)$-labeled packing of $k$ copies of $C_n$. As previously stated, for $x=1$, since $k$ is odd the construction gives us the same decomposition of $K_n$ as the one presented by Walecki \cite{alspach2008wonderful}. The idea of the decomposition is that all the edges that do not involve any fixed point are linking two vertices whose numbers have a unique difference. Thus, those edges do not overlap, and $\sigma$ defines a packing. Moreover, as $n-1$ is even, the parity of the indices of a vertex and its image are the same, as well as their labels. For $x>1$, we know that the added fixed points will not create any superposition of edges, as their neighbors have different labels and are never sent to themselves. More rigorously, we have:

\begin{itemize}
\item[*] Since $2(j-1)$ and $2k$ are even, $i+2(j-1)\mod 2k$ always has the same parity as $i$, and for any $v\not \in F$ and $j\neq j'$, $f(\sigma_j(v))=f(\sigma_{j'}(v))$. 

\bigbreak

\item[*] For every $p,q$ and every $\{u,v\},\{u',v'\}\in E_{p,q}$ , $\sigma_j(\{u,v\})=\sigma_{j'}(\{u',v'\})$ implies $j=j'$. 
\\Indeed, for $1\leq i,i'\leq k-1$, if $\sigma_j^*(\{v_i,v_{2k-1-i}\})=\sigma_{j'}^*(\{v_{i'},v_{2k-1-i'}\})$, then:
\\$\left\{
    \begin{array}{ll}
        i+2(j-1)=i'+2(j'-1)\mod 2k\\
        2k-1-i+2(j-1)=2k-1-i'+2(j'-1)\mod 2k.\\
    \end{array}
\right.$
\\By adding the two equations, since $k$ is odd, and $1\leq j,j'\leq k$, we get $j=j'$. Similarly, for $1\leq i,i'\leq k-1$, if $\sigma_j^*(\{v_i,v_{2k-i}\})=\sigma_{j'}^*(\{v_{i'},v_{2k-i'}\})$, $j=j'$.

\bigbreak

\item[*] $i+2(j-1)\mod 2k=i+2(j'-1)\mod 2k$ implies $j=j'$, so that for any $v\not \in F$ and $j\neq j'$, $\sigma_j(v) \neq \sigma_{j'}(v)$.

\bigbreak

\item[*] $k$ and $0$ have different parity, and so do $2k-i$ and $i-1$ for all $i$. Thus, for all $u,v$ such that there exists $x\in F \mbox{ such that } \{u,x\}, \{x,v\} \in E(C_n)$, $f(u)\neq f(v)$.
\end{itemize}

\bigbreak

In the following example, we present the packing of $k=3$ copies of $C_8$ with $4$ labels.

\vspace{1cm}
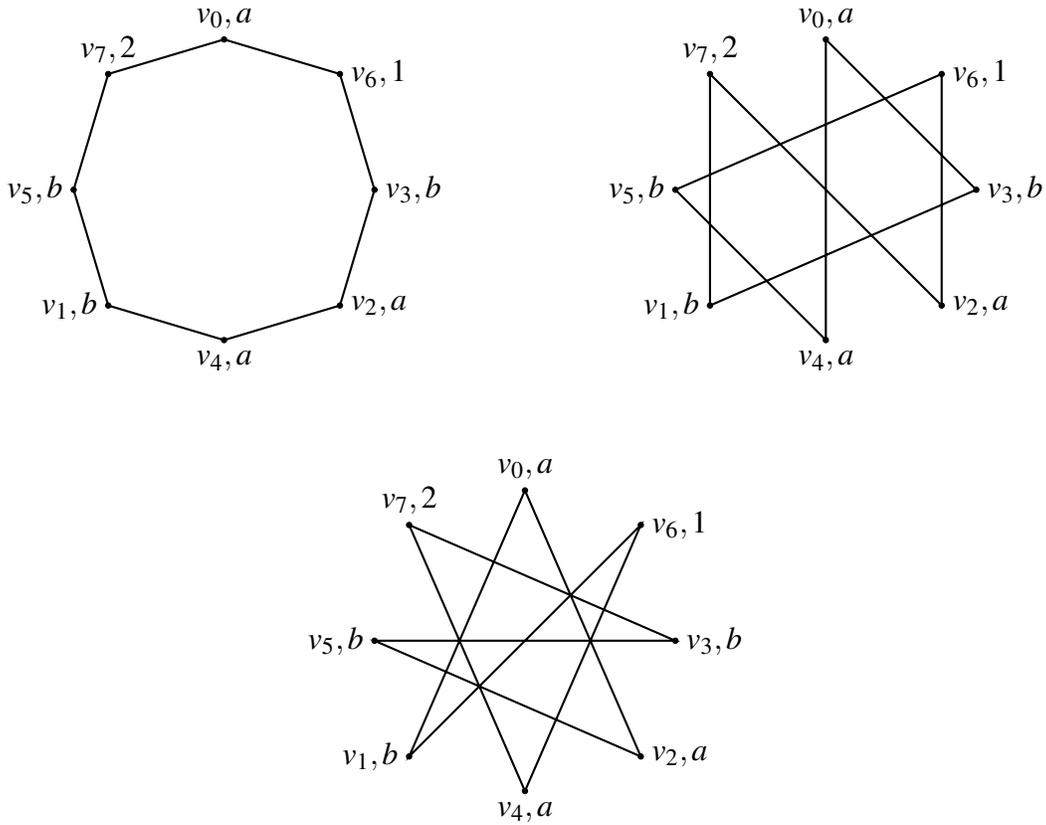
\begin{figure}[H]
	\centering
	\begin{tikzpicture}[scale=0.20]
    
    %Graphs
    
    \filldraw (10,30) circle (5pt);
    \draw (10,30) node[right]{$v_3,b$};
    \filldraw (7.7,37.7) circle (5pt);
    \draw (7.7,37.7) node[right]{$v_6,1$};
    \filldraw (0,40) circle (5pt);
    \draw (0,40) node[above]{$v_0,a$};
    \filldraw (-7.7,37.7) circle (5pt);
    \draw (-7.7,37.7) node[above]{$v_7,2$};
    \filldraw (-10,30) circle (5pt);
    \draw (-10,30) node[left]{$v_5,b$};
    \filldraw (-7.7,22.3) circle (5pt);
    \draw (-7.7,22.3) node[left]{$v_1,b$};
    \filldraw (0,20) circle (5pt);
    \draw (0,20) node[below]{$v_4,a$};
    \filldraw (7.7,22.3) circle (5pt);
    \draw (7.7,22.3) node[right]{$v_2,a$};
    
    \draw[thick] (10,30) -- (7.7,37.7) -- (0,40) -- (-7.7,37.7) -- (-10,30) -- (-7.7,22.3) -- (0,20) -- (7.7,22.3) -- (10,30);
    
    \filldraw (50,30) circle (5pt);
    \draw (50,30) node[right]{$v_3,b$};
    \filldraw (47.7,37.7) circle (5pt);
    \draw (47.7,37.7) node[right]{$v_6,1$};
    \filldraw (40,40) circle (5pt);
    \draw (40,40) node[above]{$v_0,a$};
    \filldraw (32.3,37.7) circle (5pt);
    \draw (32.3,37.7) node[above]{$v_7,2$};
    \filldraw (30,30) circle (5pt);
    \draw (30,30) node[left]{$v_5,b$};
    \filldraw (32.3,22.3) circle (5pt);
    \draw (32.3,22.3) node[left]{$v_1,b$};
    \filldraw (40,20) circle (5pt);
    \draw (40,20) node[below]{$v_4,a$};
    \filldraw (47.7,22.3) circle (5pt);
    \draw (47.7,22.3) node[right]{$v_2,a$};

    \draw[thick] (47.7,22.3) -- (32.3,37.7) -- (32.3,22.3) -- (50,30) -- (40,40) -- (40,20) -- (30,30) -- (47.7,37.7) -- (47.7,22.3);

    \filldraw (30,0) circle (5pt);
    \draw (30,0) node[right]{$v_3,b$};
    \filldraw (27.7,7.7) circle (5pt);
    \draw (27.7,7.7) node[right]{$v_6,1$};
    \filldraw (20,10) circle (5pt);
    \draw (20,10) node[above]{$v_0,a$};
    \filldraw (12.3,7.7) circle (5pt);
    \draw (12.3,7.7) node[above]{$v_7,2$};
    \filldraw (10,0) circle (5pt);
    \draw (10,0) node[left]{$v_5,b$};
    \filldraw (12.3,-7.7) circle (5pt);
    \draw (12.3,-7.7) node[left]{$v_1,b$};
    \filldraw (20,-10) circle (5pt);
    \draw (20,-10) node[below]{$v_4,a$};
    \filldraw (27.7,-7.7) circle (5pt);
    \draw (27.7,-7.7) node[right]{$v_2,a$};

    \draw[thick] (20,-10) -- (12.3,7.7) -- (30,0) -- (10,0) -- (27.7,-7.7) -- (20,10) -- (12.3,-7.7) -- (27.7,7.7) -- (20,-10);
    
	\end{tikzpicture}
    \caption{A $4$-labeled packing of $k=3$ copies of $C_{8}$.}
	\label{plusfp}
\end{figure}

\item[b)] $k\leq x\leq 2k-1$:

\bigbreak

Here, for $x=k$ and thus $n=3k$, we construct the first copy of $C_n$ by linking successively a fixed point, a vertex of label $a$, and a vertex of label $b$. Then, we construct the $k-1$ other copies by rotating $j$ times clockwise, $1\leq j \leq k-1$, the positions of the vertices of label $a$, and $j$ times anti-clockwise the positions of the vertices of label $b$. For $x>k$, we add fixed points to the previous construction on $x-k$ of the $k$ edges that link a vertex of label $a$ to a vertex of label $b$. We thus take :
 
\bigbreak

$f(v_i)=\left\{
    \begin{array}{ll}
        \frac{i}{3}+1 & \mbox{ if } i\leq 3k \mbox{ and } i \equiv 0\mod 3\\
        a & \mbox{ if } i\leq 3k \mbox{ and } i \equiv 1\mod 3\\
        b & \mbox{ if } i\leq 3k \mbox{ and } i \equiv 2\mod 3\\
        i-2k+1 & \mbox{ if } i>3k.\\
    \end{array}
\right.
$

\bigbreak

$\sigma_j(v_i)=
\left\{
    \begin{array}{ll}
    	v_{i} & \mbox{ if } i \equiv 0\mod 3 \mbox{ or } i>3k\\
        v_{i+3(j-1)\mod 3k} & \mbox{ if } i\leq 3k \mbox{ and } i \equiv 1\mod 3\\
        v_{i-3(j-1)\mod 3k} & \mbox{ if } i\leq 3k \mbox{ and } i \equiv 2\mod 3.\\
    \end{array}
\right.$

\bigbreak 

$F=\{v_i,i \equiv 0\mod 3 \mbox{ or } i>3k\}$
\\$V_a=\{v_i,i\leq 3k \mbox{ and } i \equiv 1\mod 3\}$
\\$V_b=\{v_i,i\leq 3k \mbox{ and } i \equiv 2\mod 3\}$.

\bigbreak

$E_F=\{\{v_{i-1 \mod 3k},v_{i}\}\cup \{v_{i},v_{i+1 \mod 3k}\}, 0\leq i< 3k \mbox{ and } i\equiv 0\mod 3\}\cup 
\\ \{\{v_{3i-9k+1},v_{i}\}\cup\{v_i,v_{3i-9k+2}\}, 3k\leq i< n\}$
\\$E_{a,a}=E_{b,b}=\emptyset$
\\$E_{a,b}=\{\{v_{3i+1},v_{3i+2}\}, x-k\leq i\leq k-1\}$.

\bigbreak

For $x=k$, we have $E(G)=\{\{v_i,v_{i+1 \mod n}\}, 0\leq i\leq n-1\}$, so that $G \simeq C_n$. For $x>k$, the extra fixed point are inserted in some of the edges of $C_{3k}$, and we keep $G \simeq C_n$.

\bigbreak

With the way we defined $\sigma$ and $f$, we obviously have that all the vertices are always sent to vertices that have the same label. Now, to show that $\sigma$ defines a packing, the idea is that the edges involving fixed points will never create any overlap since their neighbors have different labels and are never sent to themselves, and the other edges either, precisely because $k$ is odd. The full proof is given next:

\bigbreak

\begin{itemize}
\item[*] Since $3(j-1)$ and $3k$ are multiples of $3$, $i+3(j-1)\mod 3k$ always has the same remainder in the euclidean division by $3$ as $i$, and so does $i-3(j-1)\mod 3k$. Thus, for all $v\not \in F$ and $j\neq j'$, $f(\sigma_j(v))=f(\sigma_{j'}(v))$. 

\bigbreak

\item[*] For any $\{u,v\},\{u',v'\}\in E_{a,b}$ , $\sigma_j(\{u,v\})=\sigma_{j'}(\{u',v'\})$ implies $j=j'$. 
\\Indeed, for $x-k\leq i\leq k-1$, if $\sigma_j^*(\{v_{3i+1},v_{3i+2}\})=\sigma_{j'}^*(\{v_{3i'+1},v_{3i'+2}\})$, then:
\\$\left\{
    \begin{array}{ll}
        3i+1+3(j-1)\mod 3k=3i'+1+3(j'-1)\mod 3k\\
        3i+2-3(j-1)\mod 3k=3i'+2-3(j'-1)\mod 3k.\\
    \end{array}
\right.$
\\By substracting the two equations, since $k$ is odd, and $1\leq j,j'\leq k$, we get $j=j'$.

\bigbreak

\item[*] $i+3(j-1)\mod 3k=i+3(j'-1)\mod 3k$ implies $j=j'$, and \\$i-3(j-1)\mod 3k=i-3(j'-1)\mod 3k$ implies $j=j'$.
\\Thus, for all $v \not \in F$ and $j\neq j'$, $\sigma_j(v) \neq \sigma_{j'}(v)$.

\bigbreak

\item[*] Since for all $i$, $i-1$ and $i+1$ have have a different remainder in the euclidean \\division by $3$, and so do $3i-9k+1$ and $3i-9k+2$, we have that for all $u,v$ such that there exists $x\in F \mbox{ such that } \{u,x\}, \{x,v\} \in E(C_n)$, $f(u)\neq f(v)$.
\end{itemize}

\bigbreak

In the following figure, we give an example of the construction for $k=3$ and $n=11$.

\vspace{1cm}
\begin{figure}[H]
	\centering
	\begin{tikzpicture}[scale=0.23]
    
    %Graphs
    
    \filldraw (10,0) circle (5pt);
    \draw (10,0) node[right]{$v_2,b$};
    \filldraw (8.4,5.4) circle (5pt);
    \draw (8.4,5.4) node[right]{$v_9,4$};
    \filldraw (4.2,9.1) circle (5pt);
    \draw (4.2,9.1) node[above]{$v_1,a$};
    \filldraw (-1.4,9.8) circle (5pt);
    \draw (-1.4,9.8) node[above]{$v_0,1$};
    \filldraw (-6.5,7.6) circle (5pt);
    \draw (-6.5,7.6) node[left]{$v_8,b$};
    \filldraw (-9.6,2.8) circle (5pt);
    \draw (-9.6,2.8) node[left]{$v_7,a$};
    \filldraw (-9.6,-2.8) circle (5pt);
    \draw (-9.6,-2.8) node[left]{$v_6,3$};
    \filldraw (-6.5,-7.6) circle (5pt);
    \draw (-6.5,-7.6) node[below]{$v_5,b$};
    \filldraw (-1.4,-9.8) circle (5pt);
    \draw (-1.4,-9.8) node[below]{$v_{10},5$};
    \filldraw (4.2,-9.1) circle (5pt);
    \draw (4.2,-9.1) node[below]{$v_4,a$};
    \filldraw (8.4,-5.4) circle (5pt);
    \draw (8.4,-5.4) node[right]{$v_3,2$};
    
    \draw[thick] (10,0) -- (8.4,5.4) -- (4.2,9.1) -- (-1.4,9.8) -- (-6.5,7.6) -- (-9.6,2.8) -- (-9.6,-2.8) -- (-6.5,-7.6) -- (-1.4,-9.8) -- (4.2,-9.1) -- (8.4,-5.4) -- (10,0);
    
    \filldraw (50,0) circle (5pt);
    \draw (50,0) node[right]{$v_2,b$};
    \filldraw (48.4,5.4) circle (5pt);
    \draw (48.4,5.4) node[right]{$v_9,4$};
    \filldraw (44.2,9.1) circle (5pt);
    \draw (44.2,9.1) node[above]{$v_1,a$};
    \filldraw (38.6,9.8) circle (5pt);
    \draw (38.6,9.8) node[above]{$v_0,1$};
    \filldraw (33.5,7.6) circle (5pt);
    \draw (33.5,7.6) node[left]{$v_8,b$};
    \filldraw (30.4,2.8) circle (5pt);
    \draw (30.4,2.8) node[left]{$v_7,a$};
    \filldraw (30.4,-2.8) circle (5pt);
    \draw (30.4,-2.8) node[left]{$v_6,3$};
    \filldraw (33.5,-7.6) circle (5pt);
    \draw (33.5,-7.6) node[below]{$v_5,b$};
    \filldraw (38.6,-9.8) circle (5pt);
    \draw (38.6,-9.8) node[below]{$v_{10},5$};
    \filldraw (44.2,-9.1) circle (5pt);
    \draw (44.2,-9.1) node[below]{$v_4,a$};
    \filldraw (48.4,-5.4) circle (5pt);
    \draw (48.4,-5.4) node[right]{$v_3,2$};
    
    \draw[thick] (38.6,9.8) -- (44.2,-9.1) -- (48.4,5.4) -- (33.5,7.6) -- (48.4,-5.4) -- (30.4,2.8) -- (38.6,-9.8) -- (50,0) -- (30.4,-2.8) -- (44.2,9.1) -- (33.5,-7.6) -- (38.6,9.8);
    
    \filldraw (30,-30) circle (5pt);
    \draw (30,-30) node[right]{$v_2,b$};
    \filldraw (28.4,-24.6) circle (5pt);
    \draw (28.4,-24.6) node[right]{$v_9,4$};
    \filldraw (24.2,-20.9) circle (5pt);
    \draw (24.2,-20.9) node[above]{$v_1,a$};
    \filldraw (18.6,-20.2) circle (5pt);
    \draw (18.6,-20.2) node[above]{$v_0,1$};
    \filldraw (13.5,-22.4) circle (5pt);
    \draw (13.5,-22.4) node[left]{$v_8,b$};
    \filldraw (10.4,-27.2) circle (5pt);
    \draw (10.4,-27.2) node[left]{$v_7,a$};
    \filldraw (10.4,-32.8) circle (5pt);
    \draw (10.4,-32.8) node[left]{$v_6,3$};
    \filldraw (13.5,-37.6) circle (5pt);
    \draw (13.5,-37.6) node[below]{$v_5,b$};
    \filldraw (18.6,-39.8) circle (5pt);
    \draw (18.6,-39.8) node[below]{$v_{10},5$};
    \filldraw (24.2,-39.1) circle (5pt);
    \draw (24.2,-39.1) node[below]{$v_4,a$};
    \filldraw (28.4,-35.4) circle (5pt);
    \draw (28.4,-35.4) node[right]{$v_3,2$};
    
    \draw[thick] (18.6,-20.2) -- (10.4,-27.2) -- (28.4,-24.6) -- (13.5,-37.6) -- (28.4,-35.4) -- (24.2,-20.9) -- (18.6,-39.8) -- (13.5,-22.4) -- (10.4,-32.8) -- (24.2,-39.1) -- (30,-30) -- (18.6,-20.2);
    
	\end{tikzpicture}
    \caption{A $7$-labeled packing of $k=3$ copies of $C_{11}$.}
	\label{3k}
\end{figure}
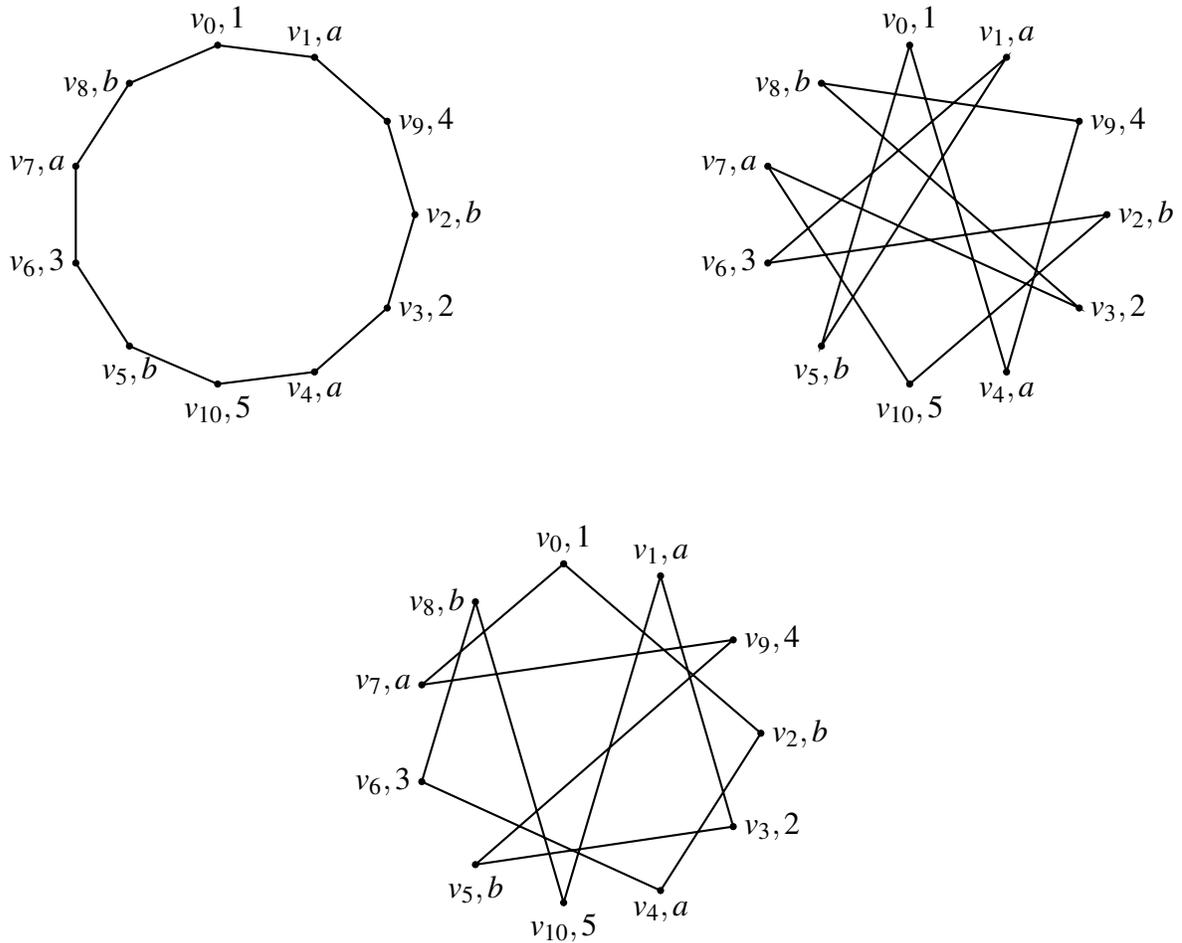

This concludes the case where $k$ is odd. 

\end{itemize}

\item[2.] $k$ is even and $x=1$:

\bigbreak

In this case, because $x=1$, $n$ is odd, and we can use a labeling of Walecki's decomposition of $K_n$ \cite{alspach2008wonderful}: We label $v_{2k}$, the fixed point, with label $1$, and every other vertex with label $b$. Then, we create the first copy of $C_n$ by joining the vertices in the following order: $$v_0,v_{n-2},v_1,v_{n-3},v_2,v_{n-4},...,v_{\frac{n-1}{2}-1},v_{\frac{n-1}{2}},v_{2k},v_0.$$ We construct the other copies by following the same order but with every number added each time to 1 modulo $n-1$. The corresponding expressions of $f$, $\sigma$, and the corresponding partitions are:

\bigbreak

$f(v_i)=\left\{
    \begin{array}{ll}
        1 & \mbox{ if } i=2k\\
        a & \mbox{ otherwise. }\\
    \end{array}
\right.
$

\bigbreak

$\sigma_j(v_i)=
\left\{
    \begin{array}{ll}
    	v_{i} & \mbox{ if } i=2k\\
        v_{i+j-1\mod 2k} & \mbox{ otherwise. }\\
    \end{array}
\right.$

\bigbreak

$F=\{v_{2k}\}$
\\$V_a=\{v_i, 0\leq i\leq 2k-1\}$.

\bigbreak

$E_F=\{\{v_{0},v_{2k}\}\cup \{v_{k},v_{2k}\}\}$
\\$E_{a,a}=\{\{v_{i},v_{2k-i}\}, 1\leq i\leq k-1\}\cup \{\{v_{i},v_{2k-1-i}\}, 0\leq i\leq k-1\}$.

\bigbreak

The obtained graph $G$ is isomorphic to $C_n$. Indeed, $G=\tau(H)$, with $\tau$ the permutation defined as follows:

$ \tau(v_i)=
\left\{
    \begin{array}{ll}
        v_{\frac{i}{2}} & \mbox{ if } i<2k \mbox{ and } i \equiv 0\mod 2\\
        v_{2k-\frac{i+1}{2}} & \mbox{ if } i<2k \mbox{ and } i \equiv 1\mod 2\\
        v_{i} & \mbox{ if } i=2k.\\
    \end{array}
\right.$

\bigbreak

We now show that $f$ and $\sigma$ define a $2$-labeled packing of $k$ copies of $C_n$. Intuitively, the construction obviously gives us the same decomposition of $K_n$ as the one presented by Walecki \cite{alspach2008wonderful}, and the labels are all the same except for the fixed point. Still, we give a more rigorous proof:

\begin{itemize}
\item[*] Since $i+j-1\mod 2k<2k$, we have, for all $v\not \in F$, $f(\sigma_j(v))=a=f(\sigma_{j'}(v))$. 

\bigbreak

\item[*] Similarly to the case 1.a), for all $p,q$ and $\{u,v\},\{u',v'\}\in E_{p,q}$ , $\sigma_j(\{u,v\})=\sigma_{j'}(\{u',v'\})$ implies $j=j'$.

\bigbreak

\item[*] Similarly to 1.a), for all $v\not \in F$ and $j\neq j'$, $\sigma_j(v) \neq \sigma_{j'}(v)$.

\bigbreak

\item[*] For all $j\neq j'$, with $1\leq j,j'\leq k$, $0+j-1\mod 2k\neq k+j'-1\mod 2k$. Thus, for all $u,v$ such that there exists $x\in F \mbox{ such that } \{u,x\}, \{x,v\} \in E(C_n)$, if $f(u)=f(v)$, then, $\sigma_j(u)\neq \sigma_{j'}(v)$ and $\sigma_j(v)\neq \sigma_{j'}(u)$.
\end{itemize}

\bigbreak

The following example, for $k=4$, gives an idea of the general construction for the case where $k$ is even and $x=1$. We can clearly see here how it corresponds to a labelling of Walecki's construction \cite{alspach2008wonderful}.
 
\vspace{1cm}
\begin{figure}[H]
	\centering
	\begin{tikzpicture}[scale=0.2]
    
    %Graphs
    
    \filldraw (50,40) circle (5pt);
    \draw (50,40) node[right]{$v_3,a$};
    \filldraw (47.1,47.1) circle (5pt);
    \draw (47.1,47.1) node[right]{$v_2,a$};
    \filldraw (40,50) circle (5pt);
    \draw (40,50) node[above]{$v_1,a$};
    \filldraw (32.9,47.1) circle (5pt);
    \draw (32.9,47.1) node[above]{$v_0,a$};
    \filldraw (30,40) circle (5pt);
    \draw (30,40) node[left]{$v_7,a$};
    \filldraw (32.9,32.9) circle (5pt);
    \draw (32.9,32.9) node[below]{$v_6,a$};
    \filldraw (40,30) circle (5pt);
    \draw (40,30) node[below]{$v_5,a$};
    \filldraw (47.1,32.9) circle (5pt);
    \draw (47.1,32.9) node[right]{$v_4,a$};  
    \filldraw (40,41) circle (5pt);
    \draw (40.2,41.2) node[above]{$v_8,1$};
    
    \draw[thick] (32.9,47.1) -- (30,40) -- (40,50) -- (32.9,32.9) -- (47.1,47.1) -- (40,30) -- (50,40) -- (47.1,32.9) -- (40,41) -- (32.9,47.1);
    
     \filldraw (90,40) circle (5pt);
    \draw (90,40) node[right]{$v_3,a$};
    \filldraw (87.1,47.1) circle (5pt);
    \draw (87.1,47.1) node[right]{$v_2,a$};
    \filldraw (80,50) circle (5pt);
    \draw (80,50) node[above]{$v_1,a$};
    \filldraw (72.9,47.1) circle (5pt);
    \draw (72.9,47.1) node[above]{$v_0,a$};
    \filldraw (70,40) circle (5pt);
    \draw (70,40) node[left]{$v_7,a$};
    \filldraw (72.9,32.9) circle (5pt);
    \draw (72.9,32.9) node[below]{$v_6,a$};
    \filldraw (80,30) circle (5pt);
    \draw (80,30) node[below]{$v_5,a$};
    \filldraw (87.1,32.9) circle (5pt);
    \draw (87.1,32.9) node[right]{$v_4,a$};  
    \filldraw (80,41) circle (5pt);
    \draw (80,41.2) node[right]{$v_8,1$};
    
    \draw[thick] (80,50) -- (72.9,47.1) -- (87.1,47.1) -- (70,40) -- (90,40) -- (72.9,32.9) -- (87.1,32.9) -- (80,30) -- (80,41) -- (80,50);
    
    \filldraw (50,10) circle (5pt);
    \draw (50,10) node[right]{$v_3,a$};
    \filldraw (47.1,17.1) circle (5pt);
    \draw (47.1,17.1) node[right]{$v_2,a$};
    \filldraw (40,20) circle (5pt);
    \draw (40,20) node[above]{$v_1,a$};
    \filldraw (32.9,17.1) circle (5pt);
    \draw (32.9,17.1) node[above]{$v_0,a$};
    \filldraw (30,10) circle (5pt);
    \draw (30,10) node[left]{$v_7,a$};
    \filldraw (32.9,2.9) circle (5pt);
    \draw (32.9,2.9) node[below]{$v_6,a$};
    \filldraw (40,0) circle (5pt);
    \draw (40,0) node[below]{$v_5,a$};
    \filldraw (47.1,2.9) circle (5pt);
    \draw (47.1,2.9) node[right]{$v_4,a$};  
    \filldraw (40,11) circle (5pt);
    \draw (39.7,11.3) node[above]{$v_8,1$};
    
    \draw[thick] (47.1,17.1) -- (50,10) -- (40,20) -- (47.1,2.9) -- (32.9,17.1) -- (40,0) -- (30,10) -- (32.9,2.9) -- (40,11) -- (47.1,17.1);
    
     \filldraw (90,10) circle (5pt);
    \draw (90,10) node[right]{$v_3,a$};
    \filldraw (87.1,17.1) circle (5pt);
    \draw (87.1,17.1) node[right]{$v_2,a$};
    \filldraw (80,20) circle (5pt);
    \draw (80,20) node[above]{$v_1,a$};
    \filldraw (72.9,17.1) circle (5pt);
    \draw (72.9,17.1) node[above]{$v_0,a$};
    \filldraw (70,10) circle (5pt);
    \draw (70,10) node[left]{$v_7,a$};
    \filldraw (72.9,2.9) circle (5pt);
    \draw (72.9,2.9) node[below]{$v_6,a$};
    \filldraw (80,0) circle (5pt);
    \draw (80,0) node[below]{$v_5,a$};
    \filldraw (87.1,2.9) circle (5pt);
    \draw (87.1,2.9) node[right]{$v_4,a$};  
    \filldraw (81,11) circle (5pt);
    \draw (82.2,11) node[above]{$v_8,1$};
    
    \draw[thick] (90,10) -- (87.1,2.9) -- (87.1,17.1) -- (80,0) -- (80,20) -- (72.9,2.9) -- (72.9,17.1) --  (70,10) -- (81,11) -- (90,10);
    
	\end{tikzpicture}
    \caption{A $2$-labeled packing of $k=4$ copies of $C_{9}$.}
	\label{k2}
\end{figure}
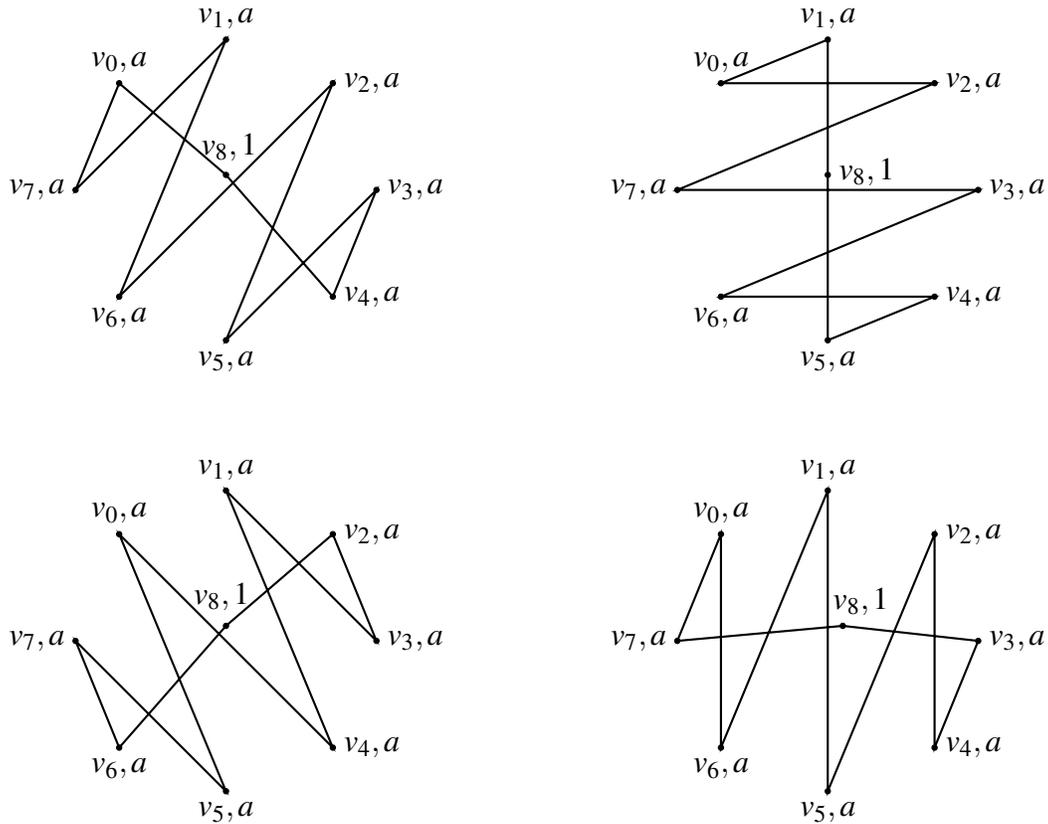

\bigbreak

\item[3.] $k$ is even and $x>1$:

\bigbreak

\begin{itemize}[leftmargin=*]

\bigbreak

\item[a)] $2\leq x\leq k-1$:

\bigbreak

The construction for this case seems more complicated, but it was actually inspired by the ones following Walecki's construction \cite{alspach2008wonderful}. For $x=2$, the first copy of $C_n$ follows the following sequence of vertices:

\bigbreak

$C_n=
\left\{
    \begin{array}{ll}
        (v_0,v_2,v_{n-4},v_4,v_{n-6},v_6...v_{\frac{k}{2}},v_{n-2},v_1,v_3,v_{n-3},v_5,v_{n-5}...& \\
        v_{\frac{k}{2}+1},v_{\frac{3k}{2}},v_{\frac{k}{2}+3},v_{\frac{3k}{2}-2},v_{\frac{k}{2}+5},v_{\frac{3k}{2}-4}...v_{\frac{3k}{2}+1},v_{n-1},v_0) & \mbox{ if } k\equiv 0\mod 4\\
        (v_0,v_2,v_{n-4},v_4,v_{n-6},v_6...v_{\frac{3k}{2}+1},v_{n-2},v_{n-3},v_1,v_{n-5},v_3...&\\
        v_{\frac{3k}{2}},v_{\frac{k}{2}+1},v_{\frac{3k}{2}-2},v_{\frac{k}{2}+3},v_{\frac{3k}{2}-4},v_{\frac{k}{2}+5}...v_{\frac{k}{2}},v_{n-1},v_0) & \mbox{ if } k\equiv 2\mod 4.\\
    \end{array}
\right.$

\bigbreak

We label $v_{2k}$ and $v_{2k+1}$, the fixed points, with respective labels $2k$ and $2k+1$, and the other vertices with $a$ if their number is even and $b$ if it is odd. Then, we create the other copies by adding each time $2$ modulo $n-2$ to every number but $2k$ and $2k+1$ in this sequence. For $x>2$, we add extra fixed points in some edges that are linking a vertex of label $a$ and a vertex of label $b$. Thus, $G \simeq C_n$. More formally, we take:

\bigbreak

$f(v_i)=\left\{
    \begin{array}{ll}
        a & \mbox{ if } 0\leq i \leq 2k-1 \mbox{ and } i \mbox{ is even }\\
        b & \mbox{ if } 0\leq i \leq 2k-1 \mbox{ and } i \mbox{ is odd }\\
        i-2k+1 & \mbox{ if } i \geq 2k.\\
    \end{array}
\right.
$

\bigbreak

$\sigma_j(v_i)=
\left\{
    \begin{array}{ll}
        v_{i+2(j-1)\mod 2k} & \mbox{ if } i < 2k\\
        v_{i} & \mbox{ if } i \geq 2k.\\
    \end{array}
\right.$

\bigbreak

$F=\{v_i,2k\leq i\leq n-1\}$
\\$V_a=\{v_i,0\leq i \leq 2k-1 \mbox{ and } i\equiv 0 \mod 2\}$
\\$V_b=\{v_i,0\leq i \leq 2k-1 \mbox{ and } i\equiv 1 \mod 2\}$.

\bigbreak

$E_F=\{\{v_{\alpha(k,i)},v_{k+3+i}\}\cup \{v_{k+3+i},v_{\alpha(k,i+1)}\}, k-1\leq i\leq k+x-4\}\cup \\ \{\{v_{\beta(k,1)},v_{2k}\},\{v_{(-1)^{\frac{k}{2}}\mod 2k},v_{2k}\}\}\cup \{\{v_0,v_{2k+1}\},\{v_{\beta(k,0)},v_{2k+1}\}\}$
\\$E_{a,a}=\{\{v_{\gamma(k,i)} ,v_{\gamma(k,i+1)}\}, 0\leq i\leq \frac{k}{2}-2\}$
\\$E_{b,b}=\{\{v_{(-1)^{\frac{k}{2}}+\gamma(k,i)\mod 2k},v_{(-1)^{\frac{k}{2}}+\gamma(k,i+1) \mod 2k} \}, 0\leq i\leq \frac{k}{2}-2\}$
\\$E_{a,b}=\{\{v_{\alpha(k,i)},v_{\alpha(k,i+1)}\}, k+x-3\leq i\leq 2k-2\}$.

\bigbreak

With:

\bigbreak

$\alpha(k,i)=(-1)^{\frac{k}{2}}(1+2\lceil \frac{\frac{k}{2}-1}{2}\rceil+(-1)^{\frac{(-1)^i+1}{2}}.2\lfloor \frac{i-k+1}{2}\rfloor)+\frac{(-1)^i+1}{2}.(k+(-1)^{\frac{k}{2}-1})\mod 2k$

$\beta(k,i)=\frac{k}{2}+\frac{(-1)^{\frac{k}{2}+i}+1}{2}(k+1) \mod 2k$

$\gamma(k,i)=2.(-1)^{i}\lceil \frac{i+1}{2}\rceil \mod 2k$.

\bigbreak

The obtained graph $G$ is isomorphic to $C_n$, as for $x=2$, $G=\tau(H)$, with $\tau$ the permutation defined as follows:

$ \tau(v_i)=
\left\{
    \begin{array}{ll}
        v_{2.(-1)^{i+1}\lceil \frac{i}{2}\rceil \mod 2k} & \mbox{ if } 0\leq i\leq \frac{k}{2}-1\\
        v_{(-1)^{\frac{k}{2}}+2.(-1)^{i+1-\frac{k}{2}}\lceil \frac{i-\frac{k}{2}}{2}\rceil \mod 2k} & \mbox{ if } \frac{k}{2}\leq i\leq k-1\\
        v_{\alpha(k,i)} & \mbox{ if } k\leq i\leq 2k-1\\
        v_{\frac{k}{2}} & \mbox{ if } i=2k\\
        v_{i} & \mbox{ if } i=2k+1.\\
    \end{array}
\right.$

For $x>2$, the extra fixed points are inserted in some edges of $C_{2k+2}$ who are linking a vertex of label $a$ and a vertex of label $b$. Thus, $G \simeq C_n$.

\bigbreak

Since $2k$ is even, all the vertices are always sent to vertices that have the same label. The neighbors of the fixed points always have different labels and are never sent to themselves, so that their edges will not create any problem for the packing. To see that the other edges do not overlap either, the principle is similar to the one of Walecki's construction \cite{alspach2008wonderful}. Indeed, for every type of edge, or differently said, for every couple of labels an edge can link, we chose the difference of the numbers of the vertices to be unique modulo $2k$ in the sequence, so that no other edge can overlap it. To be more precise, we show that the four previously mentioned properties hold:

\bigbreak

\begin{itemize}
\item[*] Similarly to the case 1.a), for all $v\not \in F$ and $j\neq j'$, $f(\sigma_j(v))=f(\sigma_{j'}(v))$. 

\bigbreak

\item[*] For all $p,q$ and for all $\{u,v\},\{u',v'\}\in E_{p,q}$ , $\sigma_j(\{u,v\})=\sigma_{j'}(\{u',v'\}) $ implies $j=j'$. 
\\Indeed, for $p=q=a$, for $0\leq i\leq \frac{k}{2}-2$, suppose that:
$$\sigma_j^*(\{v_{\gamma(k,i)},v_{\gamma(k,i+1)}\})=\sigma_{j'}^*(\{v_{\gamma(k,i')},v_{\gamma(k,i'+1)}\}).$$
Then, by adding the two equations resulting from the vertices, we get:
\small
$$2((-1)^{i+1}\lceil \frac{i}{2}\rceil+(-1)^{i}\lceil \frac{i+1}{2}\rceil) +2j \equiv 2((-1)^{i+1}\lceil \frac{i}{2}\rceil+(-1)^{i'}\lceil \frac{i'+1}{2}\rceil) +2j' \mod 2k.$$
\normalsize
\\Since $k$ is even, this gives $j=j'$. The case where $p=q=b$ is similar. 
\\For $p=a$ and $q=b$, for $k+x-3\leq i\leq 2k-2$, suppose that $$\sigma_j^*(\{v_{\alpha(k,i)},v_{\alpha(k,i+1)}\})=\sigma_{j'}^*(\{v_{\alpha(k,i')},v_{\alpha(k,i'+1)}\}).$$
Then, by adding the two equations, we get:

\bigbreak

$2.((-1)^{\frac{(-1)^i+1}{2}}\lfloor\frac{i-k+1}{2}\rfloor + (-1)^{\frac{(-1)^{i+1}+1}{2}}\lfloor\frac{i-k+2}{2}\rfloor) +4j \equiv$
\\$2.((-1)^{\frac{(-1)^{i'}+1}{2}}\lfloor\frac{i'-k+1}{2}\rfloor + (-1)^{\frac{(-1)^{i'+1}+1}{2}}\lfloor\frac{i'-k+2}{2}\rfloor) +4j'\mod 2k.$

\bigbreak

Since $k$ is even, this also gives $j=j'$.

\bigbreak

\item[*] Similarly to 1.a), for all $v\not \in F$ and $j\neq j'$, $\sigma_j(v) \neq \sigma_{j'}(v)$.

\bigbreak

\item[*] Since $k+(-1)^{\frac{k}{2}-1}$ is odd, $\alpha(k,i)$ and $\alpha(k,i+1)$ have different parity for all $i$ such that $k-1\leq i\leq k+x-4$. Plus, $\beta(k,1)$ is even while $(-1)^{\frac{k}{2}}\mod 2k$ is odd, and $0$ is even while $\beta(k,0)$ is odd. Thus, for all $u,v$ such that there exists $x\in F \mbox{ such that } \{u,x\}, \{x,v\} \in E(C_n)$, $f(u)\neq f(v)$.
\end{itemize}

\bigbreak

This construction is illustrated in the following figure, for $k=4$  and $n=10$.

\bigbreak

\vspace{1cm}
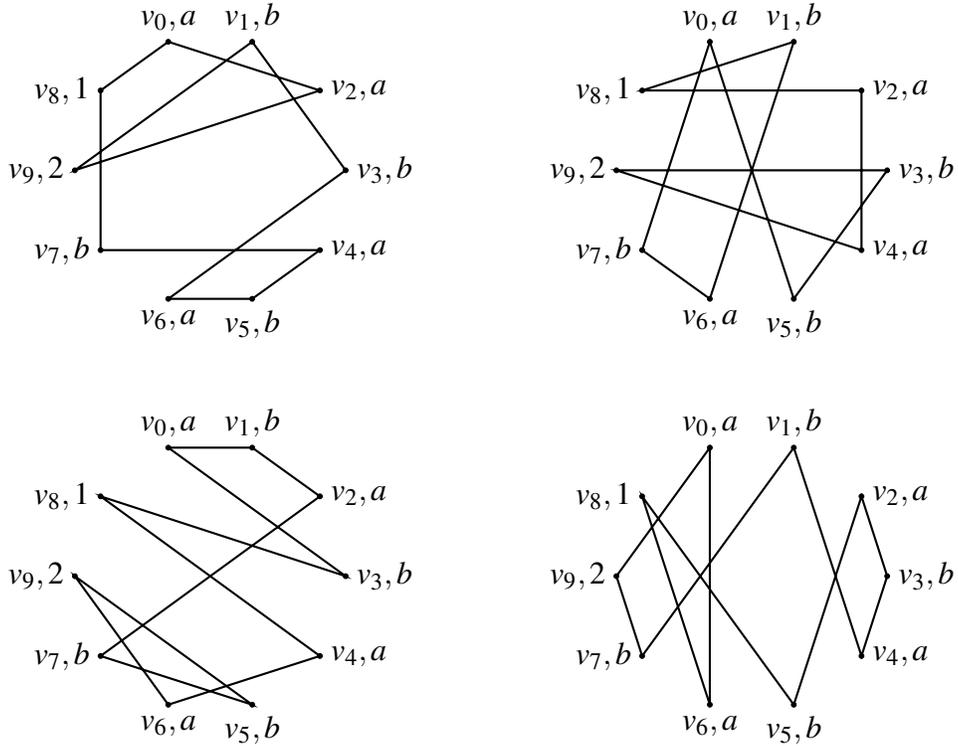
\begin{figure}[H]
	\centering
	\begin{tikzpicture}[scale=0.18]
    
    %Graphs
    
    \filldraw (50,40) circle (5pt);
    \draw (50,40) node[right]{$v_3,b$};
    \filldraw (48.1,45.9) circle (5pt);
    \draw (48.1,45.9) node[right]{$v_2,a$};
    \filldraw (43.1,49.5) circle (5pt);
    \draw (43.1,49.5) node[above]{$v_1,b$};
    \filldraw (36.9,49.5) circle (5pt);
    \draw (36.9,49.5) node[above]{$v_0,a$};
    \filldraw (31.9,45.9) circle (5pt);
    \draw (31.9,45.9) node[left]{$v_8,1$};
    \filldraw (30,40) circle (5pt);
    \draw (30,40) node[left]{$v_9,2$};
    \filldraw (31.9,34.1) circle (5pt);
    \draw (31.9,34.1) node[left]{$v_7,b$};
    \filldraw (36.9,30.5) circle (5pt);
    \draw (36.9,30.5) node[below]{$v_6,a$};
    \filldraw (43.1,30.5) circle (5pt);
    \draw (43.1,30.5) node[below]{$v_5,b$};
    \filldraw (48.1,34.1) circle (5pt);
    \draw (48.1,34.1) node[right]{$v_4,a$};
    
    \draw[thick] (36.9,49.5) -- (48.1,45.9) -- (30,40) -- (43.1,49.5) -- (50,40) -- (36.9,30.5) -- (43.1,30.5) -- (48.1,34.1) -- (31.9,34.1) -- (31.9,45.9) -- (36.9,49.5);
    
     \filldraw (90,40) circle (5pt);
    \draw (90,40) node[right]{$v_3,b$};
    \filldraw (88.1,45.9) circle (5pt);
    \draw (88.1,45.9) node[right]{$v_2,a$};
    \filldraw (83.1,49.5) circle (5pt);
    \draw (83.1,49.5) node[above]{$v_1,b$};
    \filldraw (76.9,49.5) circle (5pt);
    \draw (76.9,49.5) node[above]{$v_0,a$};
    \filldraw (71.9,45.9) circle (5pt);
    \draw (71.9,45.9) node[left]{$v_8,1$};
    \filldraw (70,40) circle (5pt);
    \draw (70,40) node[left]{$v_9,2$};
    \filldraw (71.9,34.1) circle (5pt);
    \draw (71.9,34.1) node[left]{$v_7,b$};
    \filldraw (76.9,30.5) circle (5pt);
    \draw (76.9,30.5) node[below]{$v_6,a$};
    \filldraw (83.1,30.5) circle (5pt);
    \draw (83.1,30.5) node[below]{$v_5,b$};
    \filldraw (88.1,34.1) circle (5pt);
    \draw (88.1,34.1) node[right]{$v_4,a$};
    
    \draw[thick] (88.1,45.9) -- (88.1,34.1) -- (70,40) -- (90,40) -- (83.1,30.5) -- (76.9,49.5) -- (71.9,34.1) -- (76.9,30.5) -- (83.1,49.5) -- (71.9,45.9) -- (88.1,45.9);
    
    \filldraw (50,10) circle (5pt);
    \draw (50,10) node[right]{$v_3,b$};
    \filldraw (48.1,15.9) circle (5pt);
    \draw (48.1,15.9) node[right]{$v_2,a$};
    \filldraw (43.1,19.5) circle (5pt);
    \draw (43.1,19.5) node[above]{$v_1,b$};
    \filldraw (36.9,19.5) circle (5pt);
    \draw (36.9,19.5) node[above]{$v_0,a$};
    \filldraw (31.9,15.9) circle (5pt);
    \draw (31.9,15.9) node[left]{$v_8,1$};
    \filldraw (30,10) circle (5pt);
    \draw (30,10) node[left]{$v_9,2$};
    \filldraw (31.9,4.1) circle (5pt);
    \draw (31.9,4.1) node[left]{$v_7,b$};
    \filldraw (36.9,0.5) circle (5pt);
    \draw (36.9,0.5) node[below]{$v_6,a$};
    \filldraw (43.1,0.5) circle (5pt);
    \draw (43.1,0.5) node[below]{$v_5,b$};
    \filldraw (48.1,4.1) circle (5pt);
    \draw (48.1,4.1) node[right]{$v_4,a$};
    
    \draw[thick] (48.1,4.1) -- (36.9,0.5) -- (30,10) -- (43.1,0.5) -- (31.9,4.1) -- (48.1,15.9) -- (43.1,19.5) -- (36.9,19.5) -- (50,10) -- (31.9,15.9) -- (48.1,4.1);
    
     \filldraw (90,10) circle (5pt);
    \draw (90,10) node[right]{$v_3,b$};
    \filldraw (88.1,15.9) circle (5pt);
    \draw (88.1,15.9) node[right]{$v_2,a$};
    \filldraw (83.1,19.5) circle (5pt);
    \draw (83.1,19.5) node[above]{$v_1,b$};
    \filldraw (76.9,19.5) circle (5pt);
    \draw (76.9,19.5) node[above]{$v_0,a$};
    \filldraw (71.9,15.9) circle (5pt);
    \draw (71.9,15.9) node[left]{$v_8,1$};
    \filldraw (70,10) circle (5pt);
    \draw (70,10) node[left]{$v_9,2$};
    \filldraw (71.9,4.1) circle (5pt);
    \draw (71.9,4.1) node[left]{$v_7,b$};
    \filldraw (76.9,0.5) circle (5pt);
    \draw (76.9,0.5) node[below]{$v_6,a$};
    \filldraw (83.1,0.5) circle (5pt);
    \draw (83.1,0.5) node[below]{$v_5,b$};
    \filldraw (88.1,4.1) circle (5pt);
    \draw (88.1,4.1) node[right]{$v_4,a$};
    
    \draw[thick] (76.9,0.5) -- (76.9,19.5) -- (70,10) -- (71.9,4.1) -- (83.1,19.5) -- (88.1,4.1) -- (90,10) --  (88.1,15.9) -- (83.1,0.5) -- (71.9,15.9) -- (76.9,0.5);
    
	\end{tikzpicture}
    \caption{A $4$-labeled packing of $k=4$ copies of $C_{10}$.}
	\label{Walecki2}
\end{figure}

\bigbreak

\item[b)] $k\leq x\leq 2k-1$:

\bigbreak

For this case, we also got inspired by Walecki's construction \cite{alspach2008wonderful}, but with a simpler construction. For $x=k$, we build a $(k+2)$-labeled packing of $G$ in the following way: We label the vertices of number $i\leq 2k$ by their number, and the other vertices with $a$ if their number is even and $b$ if their number is odd. We create the first copy of $C_n$ by joining the vertices in the following order: 

\bigbreak

$v_{2k},v_0,v_{2k+1},v_1,...,v_{3k-1},v_{k-1},v_{2k-2},v_{k+1},v_{2k-4},v_{k+3},v_{2k-6},...,v_{2k-k},v_{k+(k-1)},v_{2k}.$

\bigbreak

The other copies are created by adding $2$ modulo $n-2$ to the numbers of the $a$ and $b$ label vertices of this sequence. For $x>k$, we add fixed points on the edges linking a vertex of label $a$ to a vertex of label $b$. We have:

$f(v_i)=\left\{
    \begin{array}{ll}
        a & \mbox{ if } 0\leq i \leq 2k-1 \mbox{ and } i \mbox{ is even }\\
        b & \mbox{ if } 0\leq i \leq 2k-1 \mbox{ and } i \mbox{ is odd }\\
        i-2k+1 & \mbox{ if } i \geq 2k.\\
    \end{array}
\right.
$

\bigbreak

$\sigma_j(v_i)=
\left\{
    \begin{array}{ll}
        v_{i+2(j-1)\mod 2k} & \mbox{ if } i < 2k\\
        v_{i} & \mbox{ if } i \geq 2k.\\
    \end{array}
\right.$

\bigbreak

$F=\{v_i,2k\leq i\leq n-1\}$
\\$V_a=\{v_i,0\leq i \leq 2k-1 \mbox{ and } i\equiv 0 \mod 2\}$
\\$V_b=\{v_i,0\leq i \leq 2k-1 \mbox{ and } i\equiv 1 \mod 2\}$.

\bigbreak

$E_F=\{\{v_{2k+i},v_{i-1 \mod 2k}\}\cup \{v_{2k+i},v_{i}\}, 0\leq i\leq k-1\}\cup
\\ \{\{v_{4k-2\lfloor\frac{i}{2}\rfloor-2},v_{k+i}\},\{v_{k+i},v_{2\lfloor\frac{i}{2}\rfloor+(-1)^{i+1}-k}\}, 2k\leq i\leq k+x-1\}$
\\$E_{a,a}=E_{b,b}=\emptyset$
\\$E_{a,b}=\{\{v_{4k-2\lfloor\frac{i}{2}\rfloor-2},v_{2\lfloor\frac{i}{2}\rfloor+(-1)^{i+1}-k}\}, k+x\leq i\leq 3k-1\}$.

\bigbreak

For $x=k$, $G=\tau(H)$, with $\tau$ the permutation defined as follows, so that $G$ is isomorphic to $C_n$.

$ \tau(v_i)=
\left\{
    \begin{array}{ll}
        v_{2k+\frac{i}{2}} & \mbox{ if } 0\leq i\leq 2k-1, i\equiv 0\mod 2\\
        v_{\frac{i-1}{2}} & \mbox{ if } 0\leq i\leq 2k-1, i\equiv 1\mod 2\\
        v_{4k-i-2} & \mbox{ if } 2k\leq i\leq 3k-1,i\equiv 0\mod 2\\
        v_{i-k} & \mbox{ if } 2k\leq i\leq 3k-1,i\equiv 1\mod 2.\\
    \end{array}
\right.$

For $x>2$, the extra fixed points are inserted in some edges of $C_{3k}$.

\bigbreak

In this construction, all the fixed points have one neighbor with label $a$ and one neighbor with label $b$, that are never sent to themselves or each other. The edges involving them will thus never overlap. The other edges link a vertex of label $a$ with a vertex of label $b$, and the differences between their numbers are all distinct. Plus, as $2k$ is even, the parity of the vertices remains the same in all the copies, so do their labels. To be more precise, we have:

\bigbreak

\begin{itemize}
\item[*] Similarly to the case 1.a), for all $v\not \in F$ and $j\neq j'$, $f(\sigma_j(v))=f(\sigma_{j'}(v))$. 

\bigbreak

\item[*] For all $p,q$ and for all $\{u,v\},\{u',v'\}\in E_{p,q}$ , $\sigma_j(\{u,v\})=\sigma_{j'}(\{u',v'\})$ implies $j=j'$. 
\\Indeed, suppose, for a given $i$, $k+x\leq i\leq 3k-1$, that:
$$\sigma_j^*(\{v_{4k-2\lfloor\frac{i}{2}\rfloor-2},v_{2\lfloor\frac{i}{2}\rfloor+(-1)^{i+1}-k}\})=\sigma_{j'}^*(\{v_{4k-2\lfloor\frac{i'}{2}\rfloor-2},v_{2\lfloor\frac{i'}{2}\rfloor+(-1)^{i'+1}-k}\}).$$
Then, by adding the two resulting equations, we get:
$$4j+(-1)^{i+1}\mod 2k=4j'+(-1)^{i'+1}\mod 2k.$$
Since the parity of both sides of this equality must be the same, this gives $j=j'$.

\bigbreak

\item[*] Similarly to the case 1.a), for all $v\not \in F$ and $j\neq j'$, $\sigma_j(v) \neq \sigma_{j'}(v)$.

\bigbreak

\item[*] Since for all $i$ with $2k\leq i\leq k+x-1$ ,$4k-2\lfloor\frac{i}{2}\rfloor-2$ is even, and $2\lfloor\frac{i}{2}\rfloor+(-1)^{i+1}-k$ is odd, they have different parity, and so do $i-1\mod 2k$ and $i$ for all $i$ with $0\leq i\leq k-1$. We therefore have, for all $u,v$ such that there exists $x\in F \mbox{ such that } \{u,x\}, \{x,v\} \in E(C_n)$, $f(u)\neq f(v)$.
\end{itemize}

\bigbreak

We present this last construction in the following figure, that takes the example of $k=4$ and $n=12$

\bigbreak

\vspace{1cm}
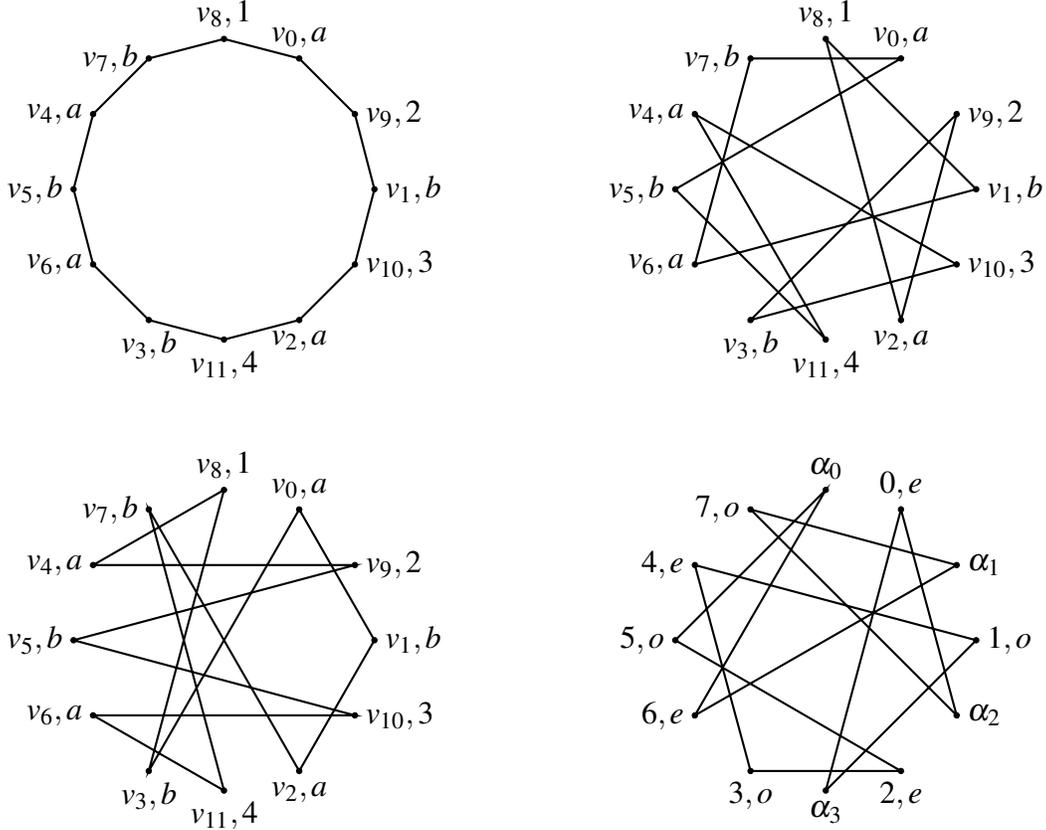
\begin{figure}[H]
	\centering
	\begin{tikzpicture}[scale=0.2]
    
%Graphs
    
    \filldraw (50,40) circle (5pt);
    \draw (50,40) node[right]{$v_1,b$};
    \filldraw (48.7,45) circle (5pt);
    \draw (48.7,45) node[right]{$v_9,2$};
    \filldraw (45,48.7) circle (5pt);
    \draw (45,48.7) node[above]{$v_0,a$};
    \filldraw (40,50) circle (5pt);
    \draw (40,50) node[above]{$v_8,1$};
    \filldraw (35,48.7) circle (5pt);
    \draw (35,48.7) node[left]{$v_7,b$};
    \filldraw (31.3,45) circle (5pt);
    \draw (31.3,45) node[left]{$v_4,a$};
    \filldraw (30,40) circle (5pt);
    \draw (30,40) node[left]{$v_5,b$};
    \filldraw (31.3,35) circle (5pt);
    \draw (31.3,35) node[left]{$v_6,a$};
    \filldraw (35,31.3) circle (5pt);
    \draw (35,31.3) node[below]{$v_3,b$};
    \filldraw (40,30) circle (5pt);
    \draw (40,30) node[below]{$v_{11},4$};
    \filldraw (45,31.3) circle (5pt);
    \draw (45,31.3) node[below]{$v_2,a$};
    \filldraw (48.7,35) circle (5pt);
    \draw (48.7,35) node[right]{$v_{10},3$};
    
    \draw[thick] (40,30) -- (45,31.3) -- (48.7,35) -- (50,40) -- (48.7,45) -- (45,48.7) -- (40,50) -- (35,48.7) -- (31.3,45) -- (30,40) -- (31.3,35) -- (35,31.3) -- (40,30);
    
    \filldraw (90,40) circle (5pt);
    \draw (90,40) node[right]{$v_1,b$};
    \filldraw (88.7,45) circle (5pt);
    \draw (88.7,45) node[right]{$v_9,2$};
    \filldraw (85,48.7) circle (5pt);
    \draw (85,48.7) node[above]{$v_0,a$};
    \filldraw (80,50) circle (5pt);
    \draw (80,50) node[above]{$v_8,1$};
    \filldraw (75,48.7) circle (5pt);
    \draw (75,48.7) node[left]{$v_7,b$};
    \filldraw (71.3,45) circle (5pt);
    \draw (71.3,45) node[left]{$v_4,a$};
    \filldraw (70,40) circle (5pt);
    \draw (70,40) node[left]{$v_5,b$};
    \filldraw (71.3,35) circle (5pt);
    \draw (71.3,35) node[left]{$v_6,a$};
    \filldraw (75,31.3) circle (5pt);
    \draw (75,31.3) node[below]{$v_3,b$};
    \filldraw (80,30) circle (5pt);
    \draw (80,30) node[below]{$v_{11},4$};
    \filldraw (85,31.3) circle (5pt);
    \draw (85,31.3) node[below]{$v_2,a$};
    \filldraw (88.7,35) circle (5pt);
    \draw (88.7,35) node[right]{$v_{10},3$};
    
    \draw[thick] (80,30) -- (71.3,45) -- (88.7,35) -- (75,31.3) -- (88.7,45) -- (85,31.3) -- (80,50) -- (90,40) -- (71.3,35) -- (75,48.7) -- (85,48.7) -- (70,40) -- (80,30);
    
        \filldraw (50,10) circle (5pt);
    \draw (50,10) node[right]{$v_1,b$};
    \filldraw (48.7,15) circle (5pt);
    \draw (48.7,15) node[right]{$v_9,2$};
    \filldraw (45,18.7) circle (5pt);
    \draw (45,18.7) node[above]{$v_0,a$};
    \filldraw (40,20) circle (5pt);
    \draw (40,20) node[above]{$v_8,1$};
    \filldraw (35,18.7) circle (5pt);
    \draw (35,18.7) node[left]{$v_7,b$};
    \filldraw (31.3,15) circle (5pt);
    \draw (31.3,15) node[left]{$v_4,a$};
    \filldraw (30,10) circle (5pt);
    \draw (30,10) node[left]{$v_5,b$};
    \filldraw (31.3,5) circle (5pt);
    \draw (31.3,5) node[left]{$v_6,a$};
    \filldraw (35,1.3) circle (5pt);
    \draw (35,1.3) node[below]{$v_3,b$};
    \filldraw (40,0) circle (5pt);
    \draw (40,0) node[below]{$v_{11},4$};
    \filldraw (45,1.3) circle (5pt);
    \draw (45,1.3) node[below]{$v_2,a$};
    \filldraw (48.7,5) circle (5pt);
    \draw (48.7,5) node[right]{$v_{10},3$};
    
    \draw[thick] (40,0) -- (31.3,5) -- (48.7,5) -- (30,10) -- (48.7,15) -- (31.3,15) -- (40,20) -- (35,1.3) -- (45,18.7) -- (50,10) -- (45,1.3) -- (35,18.7) -- (40,0);
    
    \filldraw (90,10) circle (5pt);
    \draw (90,10) node[right]{$1,o$};
    \filldraw (88.7,15) circle (5pt);
    \draw (88.7,15) node[right]{$\alpha_1$};
    \filldraw (85,18.7) circle (5pt);
    \draw (85,18.7) node[above]{$0,e$};
    \filldraw (80,20) circle (5pt);
    \draw (80,20) node[above]{$\alpha_0$};
    \filldraw (75,18.7) circle (5pt);
    \draw (75,18.7) node[left]{$7,o$};
    \filldraw (71.3,15) circle (5pt);
    \draw (71.3,15) node[left]{$4,e$};
    \filldraw (70,10) circle (5pt);
    \draw (70,10) node[left]{$5,o$};
    \filldraw (71.3,5) circle (5pt);
    \draw (71.3,5) node[left]{$6,e$};
    \filldraw (75,1.3) circle (5pt);
    \draw (75,1.3) node[below]{$3,o$};
    \filldraw (80,0) circle (5pt);
    \draw (80,0) node[below]{$\alpha_3$};
    \filldraw (85,1.3) circle (5pt);
    \draw (85,1.3) node[below]{$2,e$};
    \filldraw (88.7,5) circle (5pt);
    \draw (88.7,5) node[right]{$\alpha_2$};
    
    \draw[thick] (80,0) -- (85,18.7) -- (88.7,5) -- (75,18.7) -- (88.7,15) -- (71.3,5) -- (80,20) -- (70,10) -- (85,1.3) -- (75,1.3) -- (71.3,15) -- (90,10) -- (80,0);
    
	\end{tikzpicture}
    \caption{A $6$-labeled packing of $k=4$ copies of $C_{12}$.}
	\label{3k2}
\end{figure}

\end{itemize}
\end{itemize}

This concludes case 3), and thus the entire proof of Theorem \ref{borne_sup_lemma}.
\end{proof}

\bigbreak

Concerning the upper bound given by Conjecture \ref{THEconjecture}, we have already seen, with a counter-example, that it did not hold in general. That being said, we still found some sufficient conditions for it to hold. The following lemma helped in the process of finding those conditions:

\begin{lemma}  Let $G$ be the cycle $C_n$ of order $n=2k+x$, $k\geq 2$, $1\leq x \leq 2k-1$. Let $f$ be a $p$-labeling of $k$ copies of $G$. Let $q=\min \limits_{i\in \{1,...,p\}}|\{v\in V(C_n) : f(v)=i\}|$ be the minimum, over all $p$ labels, of the number of vertices that have this label. Then, there exists in $G$ a set of at least $2k$ vertices that are associated, together, to at most $2q$ labels.
\label{2klemma}
\end{lemma}

\begin{proof}  Let $j$ be one of the labels associated to exactly $q$ vertices, and $V_j$ be the set of vertices of $G$ with label $j$. All the vertices of $V_j$ have two neighbors. Let $N_j$ be the set of neighbors, and $M_j=N_j-(N_j\cap V_j)$ be the set of neighbors that are not in $V_j$. Let $L$ be the number of different labels $l_1$, $l_2$, ..., $l_L$ associated to the vertices of $M_j$, $V(l_1)$, $V(l_2)$, ..., $V(l_L)$ be the sets of vertices of $G$ associated to those respective labels, and $U=V(l_1)\cup V(l_2) \cup ... \cup V(l_L)$. Let $e$ be the number of edges linking two vertices from $V_j$ in $G$, and, for all $1 \leq i \leq L$, $e_i$ be the number of edges linking a vertex from $V_j$ to a vertex from $V(l_i)$ in $G$. To have a valid packing, we must have, for every $i$ with $1 \leq i \leq L, ke_i\leq q|V(l_i)|$. Therefore, we have: 
$$\sum \limits_{i=1}^L e_i \leq \frac{q}{k}\sum \limits_{i=1}^L|V(l_i)|.$$
But $\sum \limits_{i=1}^L e_i=2q-2e$, and, since the $V(l_i)$ are pairwise disjoint, $\sum \limits_{i=1}^L|V(l_i)|=|U|$. Therefore, 
$$|U| \geq \frac{k}{q}(2q-2e).$$ 
If $e=0$, we obtain $|U| \geq 2k$. Moreover, at best, in $M_j$, we have $2q$ distinct vertices that all have distinct labels, so that $L\leq 2q$. By taking the set $U$, we obtain a set of at least $2k$ vertices in $G$ associated to at most $2q$ labels. If $e\geq 1$, to have a valid packing, we must have $ke\leq \frac{q(q-1)}{2}$, we thus have $|U| \geq 2k-(q-1)$. Plus, at best, in $M_j$, we have $2q-2$ distinct vertices that all have distinct labels, so that $L\leq 2q-2$. Therefore, by taking the set $U\cup V_j$, we obtain a set of at least $2k+1$ vertices in $G$ associated to at most $2q-1$ labels. Those two cases give the result.
\end{proof}

Lemma \ref{2klemma} led, in particular, to the following theorem, that gives a sufficient condition for the upper bound given by Conjecture \ref{THEconjecture} to hold.

\begin{theorem}  Let $G$ be the cycle $C_n$ of order $n=2k+x$, $k\geq 2$, $1\leq x \leq 2k-1$. Let $f$ be a $p$-labeling of $k$ copies of $G$. Let $q=\min \limits_{i\in \{1,...,p\}}|\{v\in V(C_n) : f(v)=i\}|$ be the minimum, over all $p$ labels, of the number of vertices that have this label. If $q=1$, meaning that $f$ has fixed points, or $q\leq \frac{x}{2}$, then $p \leq x+2$.
\label{qtheorem}
\end{theorem} 

\begin{proof}  We know by Lemma \ref{2klemma} that there exists in $G$ a group of at least $2k$ vertices associated to at most $2q$ labels, and the rest of the labels are represented at least $q$ times. Therefore, we have:
$$p\leq 2q + \lfloor \frac{n-2k}{q} \rfloor \leq 2q + \frac{x}{q}.$$

\noindent But, for all $q\geq1$, $2q + \frac{x}{q} \leq x +2$.
\\Indeed, for q=1, we have $2q + \frac{x}{q}=x +2$.
\\For $q>1$, $2q + \frac{x}{q} \leq x+2 \leftrightarrow q\leq \frac{x}{2}$, which is true by hypothesis.
Therefore, we have $p\leq x+2$.
\end{proof}

The disadvantage of Theorem \ref{qtheorem} is that the conditions it gives are on a preexisting packing, while we want the upper bound to hold for all packings, with conditions on $k$ and $x$ only. Such conditions are thus given in the following theorem:

\begin{theorem}  Let $C_n$ be the cycle of order $n=2k+x$, $k\geq 2$, $1\leq x \leq 2k-1$. If $x\geq \sqrt{4k-2}$, then $\lambda^k(C_n) \leq x+2$.
\label{xtheorem}
\end{theorem} 

\begin{proof}  Let us suppose that $\lambda^k(C_n) \geq x+3$. Then, there exists a $(x+3)$-labeling $f$ of $k$ copies of $G$. Let $q$ be the minimum, over all $p$ labels, of the number of vertices that have this label. Then, by Theorem \ref{qtheorem}, we have $q\geq \frac{x+1}{2}$. We have therefore at least $x+3$ labels, each represented by at least $\frac{x+1}{2}$ vertices, so that we have:
$$\frac{(x+1)(x+3)}{2}\leq 2k+x$$
Or $$x^2+2x+(3-4k)\leq 0.$$
For $x\geq \sqrt{4k-2}$, this last inequality is false and we obtain a contradiction.
\end{proof}

We now have a satisfying sufficient condition for the upper bound of $x+2$ to be valid, even more satisfying when taking into account the fact that the proportion of cases it covers grows with $n$. That being said, this bound is not the one given in Conjecture \ref{THEconjecture} for the particular case where $k$ is even and $x=1$. The following theorem thus gives a sufficient condition for this case, that depends on the considered packing:

\begin{theorem}  Let $C_n$ be the cycle of order $n=2k+1$, $k\geq 2$, k even. Let $f$ be a $p$-labeling of $k$ copies of $C_n$. Let $q=\min \limits_{i\in \{1,...,p\}}|\{v\in V(C_n) : f(v)=i\}|$ be the minimum, over all $p$ labels, of the number of vertices that have this label. If $q=1$, then $p \leq 2.$
\label{upper_x=1}
\end{theorem}

\begin{proof}  From Theorem \ref{qtheorem}, since $q=1$, we already have $\lambda^k(C_n)\leq 3$. Let us assume that $\lambda^k(C_n)=3$. Then, there exists a $3$-labeled packing of $k$ copies of $C_n$. By Lemma \ref{2klemma}, we know that there exists in this packing a set of $2k$ vertices associated to exactly $2$ labels, called $a$ and $b$, and one fixed point. 
\\ \indent We have $k(2k+1)=\frac{(2k+1)(2k+1-1)}{2}$, so that every edge between any two vertices is going to belong to exactly one of $k$ the copies of $G$. In particular, every edge between the fixed point and the other vertices is going to belong to a copy, so that the fixed point has to be linked to exactly one vertex with label $a$, and one vertex with label $b$. Therefore, there has to be $k$ vertices with label $a$ and $k$ vertices with label $b$. The number of edges linking two vertices of label $a$ in $G$ is therefore $\frac{k(k-1)}{2k}=\frac{k-1}{2}$, which means that $k$ is odd, and contradicts the hypothesis. Therefore, $\lambda^k(C_n)\neq 3$, and $\lambda^k(C_n)\leq 2$.
\end{proof}

For the cases where the previously seen conditions are not verified, the following theorem gives a necessary condition on $p$ for a $p$-labeled packing of $C_{2k+x}$ to exist and thus provides a different upper bound on $\lambda^k(C_{2k+x})$:

\begin{theorem} Let $C_n$ be the cycle of order $n=2k+x$. For $p \in \mathbb{N}^*$, if $p\leq \lambda^k(C_n)$,  there exists a partition $n_1,n_2,...,n_p$ of $n$ into $p$ parts such that $\sum \limits_{i=1}^{p} \lfloor \frac{n_i(n_i-1)}{2k} \rfloor + \sum \limits_{i=1}^{p} \sum \limits_{j=i+1}^{p} \lfloor \frac{n_in_j}{k} \rfloor \geq n$.
\label{max_nb_edges}
\end{theorem}

\begin{proof} Since $p\leq \lambda^k(C_n)$, there exists a $p$-labeled packing $f$ of $k$ copies of $C_n$ into $K_n$. Let $1,2,...,p$ be the name of the labels of $f$, and $n_1,n_2,...,n_p$ be the number of vertices respectively associated to those labels. Obviously, $n_1,n_2,...,n_p$ is a partition of $n$ into $p$ parts. 

For all $i \in [1,p]$, $j \in [i+1,p]$, let $m_i$ be the number of edges of $C_n$ in $f$ linking two vertices of label $i$, and let $m_{i,j}$ be the number of edges of $C_n$ in $f$ linking a vertex of label $i$ and a vertex of label $j$. 

We have $\sum \limits_{i=1}^{p} m_i + \sum \limits_{i=1}^{p} \sum \limits_{j=i+1}^{p} m_{i,j} = n$. 

For $f$ to be a valid packing, we must have, for all $i \in [1,p]$, $m_i\leq \lfloor \frac{n_i(n_i-1)}{2k} \rfloor$, and, for all $j \in [i+1,p]$, $m_{i,j} \leq \sum \limits_{j=i+1}^{p} \lfloor \frac{n_in_j}{k} \rfloor$. 

Thus, $\sum \limits_{i=1}^{p} \lfloor \frac{n_i(n_i-1)}{2k} \rfloor + \sum \limits_{i=1}^{p} \sum \limits_{j=i+1}^{p} \lfloor \frac{n_in_j}{k} \rfloor \geq n$.
\end{proof}

\bigbreak

By regrouping the sufficient conditions of Theorems \ref{x2_keven}, \ref{qtheorem} and \ref{xtheorem}, we can already restrict the possible values of $k$ and $x$ for which the second inequality of Conjecture \ref{THEconjecture} might not hold. Table \ref{tab} lists those values and gives the upper bound given by Theorem \ref{max_nb_edges} for each of those, until $k=35$.

\bigbreak

\begin{table}[H]
\huge
\begin{center}
\scalebox{0.5}{
\begin{tabular}[scale=0.1]{|c|c|c|c|c|c|}
  \hline
  $k$ & $x$ & $n$ & maximal possible value of $\lambda^k(C_n)$ \\
  \hline
  $9$ & $3$ & $21$ & $x+4$ \\
  \hline
  $15$ & $6$ & $36$ & $x+3$ \\
  \hline
  $16$ & $4$ & $36$ & $x+5$ \\
  \hline
  $16$ & $5$ & $37$ & $x+3$ \\
  \hline
  $24$ & $7$ & $55$ & $x+4$ \\
  \hline
  $25$ & $5$ & $55$ & $x+6$ \\
  \hline
  $25$ & $6$ & $56$ & $x+4$ \\
  \hline
  $27$ & $3$ & $57$ & $x+4$ \\
  \hline
  $28$ & $4$ & $60$ & $x+3$ \\
  \hline
  $33$ & $4$ & $70$ & $x+3$ \\
  \hline
  $33$ & $5$ & $71$ & $x+3$ \\
  \hline
  $34$ & $10$ & $78$ & $x+3$ \\
  \hline
  $35$ & $4$ & $74$ & $x+3$ \\
  \hline
  $35$ & $8$ & $78$ & $x+5$ \\
  \hline
  $35$ & $9$ & $79$ & $x+3$ \\
  \hline
\end{tabular}}
\end{center}
\caption{Values of $k$ and $x$, until $k=35$, for which the second inequality of Conjecture \ref{THEconjecture} is unproven. The upper bound given by Theorem \ref{max_nb_edges} is also expressed, depending on $x$.
}
\label{tab}
\end{table}

\section{Labeled Packing of Circuits}

In this part, we study the directed version of the previous problem, that is the labeled packing of circuits, where the circuit $\overrightarrow{C_n}$ of order $n$ is the directed graph defined by $V(\overrightarrow{C_n})=\{v_i,0\leq i\leq n-1\}$ and $E(\overrightarrow{C_n})=\{(v_i,v_{i+1 \mod n}),0\leq i\leq n-1\}$.

\bigbreak

We define a packing of directed graphs $(H_1,H_2,...,H_k)$ in the directed graph $G$ the same way as a packing of graphs, the only difference being that for all $i$, $E(H_i)$ stands for the arcs of $H_i$, and $E(G)$ for the arcs of $G$, so that the induced images of the arcs of the $H_i$ must go into the arcs of $G$ without intersecting. We focus on the packings of $k$ copies of $H$ into the complete digraph $\overleftrightarrow{K_n}$, the digraph of order $n$ with all possible arcs. Given those modifications, the definitions of a labeled packing and of $\lambda^k(G)$ for $G$ a digraph are direct.

\bigbreak

The results we present are very similar to the ones presented for cycles, with some adaptations. We first have the following conditions for the $k$-placement to exist:

\begin{theorem} $\overrightarrow{C_{k+x}}$ admits a $k$-placement if and only if $x\geq 1$, and $(x,n)\not \in \{(1,4),(1,6)\}$.
\end{theorem} 
\begin{proof} First, for a $k$-placement of $\overrightarrow{C_n}$ to exist, we must have $k|E(\overrightarrow{C_n})|\leq |E(\overleftrightarrow{K_n})|$, so that $n\geq k+1$. Plus, if $x=1$, the $k$-placement would actually be a decomposition of $\overleftrightarrow{K_n}$, and we know from Theorem 1.1 in \cite{alspach2003cycle} that such a decomposition exists if and only if $n\neq 4$, and $n\neq 6$. Now, since the existence of a packing of $k$ copies of $\overrightarrow{C_{k+x}}$ implies the existence of a packing of $k-1$ copies of $\overrightarrow{C_{(k-1)+(x+1)}}$, and since we have the following packings of $k=2$ copies of $\overrightarrow{C_4}$ and $k=4$ copies of $\overrightarrow{C_6}$, this gives the result:
 
 \vspace{1cm}
\begin{figure}[H]
	\centering
	\begin{tikzpicture}[scale=0.2]
    
    %Graphs
    
    \filldraw (10,0) circle (5pt);
    \draw (10,0) node[right]{$v_0$};
	\filldraw (10,10) circle (5pt);
    \draw (10,10) node[above]{$v_1$};
    \filldraw (0,10) circle (5pt);
    \draw (0,10) node[above]{$v_2$};
    \filldraw (0,0) circle (5pt);
    \draw (0,0) node[left]{$v_3$};

    \draw[thick,-{Latex[length=3mm]}] (10,0) -- (10,10);
    \draw[thick,-{Latex[length=3mm]}] (10,10) -- (0,10);
    \draw[thick,-{Latex[length=3mm]}] (0,10) -- (0,0);
    \draw[thick,-{Latex[length=3mm]}] (0,0) -- (10,0);
    
    \filldraw (40,0) circle (5pt);
    \draw (40,0) node[right]{$v_0$};
	\filldraw (40,10) circle (5pt);
    \draw (40,10) node[above]{$v_1$};
    \filldraw (30,10) circle (5pt);
    \draw (30,10) node[above]{$v_2$};
    \filldraw (30,0) circle (5pt);
    \draw (30,0) node[left]{$v_3$};

    \draw[thick,-{Latex[length=3mm]}] (40,0) -- (30,0);
    \draw[thick,-{Latex[length=3mm]}] (30,0) -- (30,10);
    \draw[thick,-{Latex[length=3mm]}] (30,10) -- (40,10);
    \draw[thick,-{Latex[length=3mm]}] (40,10) -- (40,0);

	\end{tikzpicture}
    \caption{A packing of $k=2$ copies of $\vec{C_4}$.}
	\label{packing_2_6}
\end{figure}
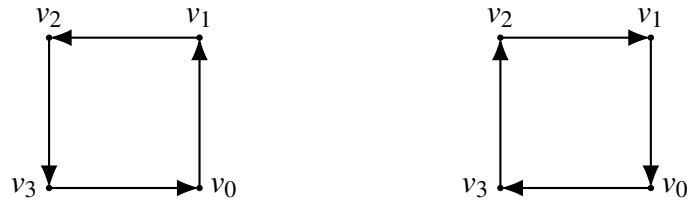

\vspace{1cm}
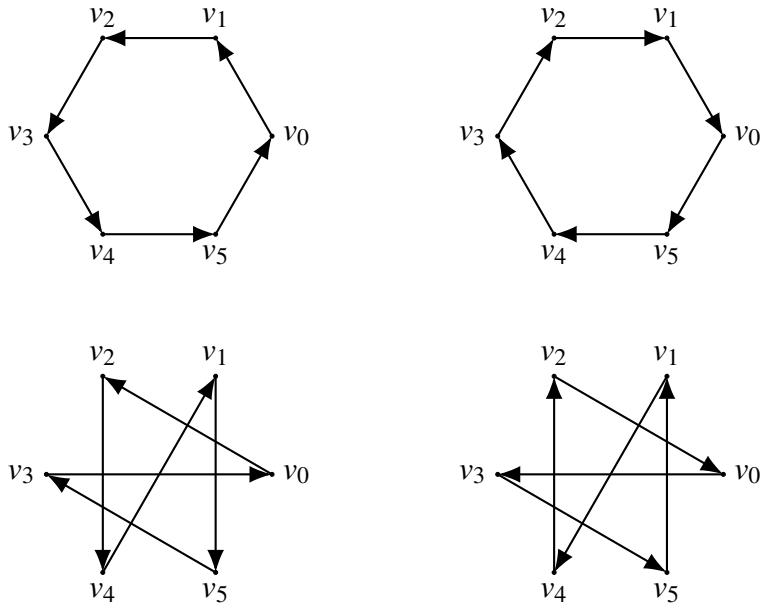
\begin{figure}[H]
	\centering
	\begin{tikzpicture}[scale=0.15]
    
    %Graphs
    
    \filldraw (10,0) circle (5pt);
    \draw (10,0) node[right]{$v_0$};
	\filldraw (5,8.7) circle (5pt);
    \draw (5,8.7) node[above]{$v_1$};
    \filldraw (-5,8.7) circle (5pt);
    \draw (-5,8.7) node[above]{$v_2$};
    \filldraw (-10,0) circle (5pt);
    \draw (-10,0) node[left]{$v_3$};
    \filldraw (-5,-8.7) circle (5pt);
    \draw (-5,-8.7) node[below]{$v_4$};
    \filldraw (5,-8.7) circle (5pt);
    \draw (5,-8.7) node[below]{$v_5$};
    \draw[thick,-{Latex[length=3mm]}] (10,0) -- (5,8.7);
    \draw[thick,-{Latex[length=3mm]}] (5,8.7) -- (-5,8.7);
    \draw[thick,-{Latex[length=3mm]}] (-5,8.7) -- (-10,0);
    \draw[thick,-{Latex[length=3mm]}] (-10,0) -- (-5,-8.7);
    \draw[thick,-{Latex[length=3mm]}] (-5,-8.7) -- (5,-8.7);
    \draw[thick,-{Latex[length=3mm]}] (5,-8.7) -- (10,0);
    
    \filldraw (50,0) circle (5pt);
    \draw (50,0) node[right]{$v_0$};
	\filldraw (45,8.7) circle (5pt);
    \draw (45,8.7) node[above]{$v_1$};
    \filldraw (35,8.7) circle (5pt);
    \draw (35,8.7) node[above]{$v_2$};
    \filldraw (30,0) circle (5pt);
    \draw (30,0) node[left]{$v_3$};
    \filldraw (35,-8.7) circle (5pt);
    \draw (35,-8.7) node[below]{$v_4$};
    \filldraw (45,-8.7) circle (5pt);
    \draw (45,-8.7) node[below]{$v_5$};
    \draw[thick,-{Latex[length=3mm]}] (50,0) -- (45,-8.7);
    \draw[thick,-{Latex[length=3mm]}] (45,-8.7) -- (35,-8.7);
    \draw[thick,-{Latex[length=3mm]}] (35,-8.7) -- (30,0);
    \draw[thick,-{Latex[length=3mm]}] (30,0) -- (35,8.7);
    \draw[thick,-{Latex[length=3mm]}] (35,8.7) -- (45,8.7);
    \draw[thick,-{Latex[length=3mm]}] (45,8.7) -- (50,0);
    
   \filldraw (10,-30) circle (5pt);
    \draw (10,-30) node[right]{$v_0$};
	\filldraw (5,-21.3) circle (5pt);
    \draw (5,-21.3) node[above]{$v_1$};
    \filldraw (-5,-21.3) circle (5pt);
    \draw (-5,-21.3) node[above]{$v_2$};
    \filldraw (-10,-30) circle (5pt);
    \draw (-10,-30) node[left]{$v_3$};
    \filldraw (-5,-38.7) circle (5pt);
    \draw (-5,-38.7) node[below]{$v_4$};
    \filldraw (5,-38.7) circle (5pt);
    \draw (5,-38.7) node[below]{$v_5$};
    \draw[thick,-{Latex[length=3mm]}] (10,-30) -- (-5,-21.3);
    \draw[thick,-{Latex[length=3mm]}] (-5,-21.3) -- (-5,-38.7);
    \draw[thick,-{Latex[length=3mm]}] (-5,-38.7) -- (5,-21.3);
    \draw[thick,-{Latex[length=3mm]}] (5,-21.3) -- (5,-38.7);
    \draw[thick,-{Latex[length=3mm]}] (5,-38.7) -- (-10,-30);
    \draw[thick,-{Latex[length=3mm]}] (-10,-30) -- (10,-30);
    
    \filldraw (50,-30) circle (5pt);
    \draw (50,-30) node[right]{$v_0$};
	\filldraw (45,-21.3) circle (5pt);
    \draw (45,-21.3) node[above]{$v_1$};
    \filldraw (35,-21.3) circle (5pt);
    \draw (35,-21.3) node[above]{$v_2$};
    \filldraw (30,-30) circle (5pt);
    \draw (30,-30) node[left]{$v_3$};
    \filldraw (35,-38.7) circle (5pt);
    \draw (35,-38.7) node[below]{$v_4$};
    \filldraw (45,-38.7) circle (5pt);
    \draw (45,-38.7) node[below]{$v_5$};
    \draw[thick,-{Latex[length=3mm]}] (50,-30) -- (30,-30);
    \draw[thick,-{Latex[length=3mm]}] (30,-30) -- (45,-38.7);
    \draw[thick,-{Latex[length=3mm]}] (45,-38.7) -- (45,-21.3);
    \draw[thick,-{Latex[length=3mm]}] (45,-21.3) -- (35,-38.7);
    \draw[thick,-{Latex[length=3mm]}] (35,-38.7) -- (35,-21.3);
    \draw[thick,-{Latex[length=3mm]}] (35,-21.3) -- (50,-30);
    
	\end{tikzpicture}
    \caption{A packing of $k=4$ copies of $\vec{C_6}$.}
	\label{packing_4_6}
\end{figure}

\end{proof}

When those conditions are satisfied, we can study the value of $\lambda^k(\overrightarrow{C_n})$. We first have the following lemma, that is an extension of Duchêne, Kheddouci, Nowakowski and Tahraoui's Lemma 7 \cite{duchene2013labeled} to the case of circuits, following the same proof:

\begin{lemma} For every circuit $\overrightarrow{C_n}$ of order $n>k$, with $(x,n)\not \in \{(1,4),(1,6)\}$, we have $\lambda^k(\overrightarrow{C_n}) \leq \lfloor \frac{n}{2} \rfloor + \lfloor \frac{\lceil \frac{n}{2} \rceil}{k} \rfloor.$
\label{borne_sup_lemma}
\end{lemma}

\noindent When $n\geq 2k$, the value of $\lambda^k(\overrightarrow{C_n})$ can be exactly found with the following theorem:

\begin{theorem}  For every circuit $\overrightarrow{C_n}$ of order $n=2km+x$, $k,m\geq 1$, $x<2k$, and $(x,n) \not \in \{(1,4),(1,6)\}$ we have $\lambda^k(\overrightarrow{C_n}) = \lfloor \frac{n}{2} \rfloor + \lfloor \frac{\lceil \frac{n}{2} \rceil}{k} \rfloor.$
\label{borne_sup_circuits}
\end{theorem}

\begin{proof}
From Lemma \ref{borne_sup_lemma}, we get the upper bound.
\\For the lower bound, we give a construction of the corresponding labeled packing of $\overrightarrow{C_n}$. The idea is to put a fixed point one vertex out of two on the first copy of $\overrightarrow{C_n}$, and to have at least $k$ vertices of any other label so that they will never be sent to themselves and thus never create superpositions of arcs.

\bigbreak

Let $V(C_n)=\{v_0,...,v_n\}$ be the set of vertices of $\overrightarrow{C_n}$ and \\$E(C_n)=\{(v_i,v_{i+1\mod n}, 0\leq i\leq n-1\}$ be its set of arcs.

\bigbreak

We label the vertices with the labeling $f$:
\\$f(v_i)=\left\{
    \begin{array}{ll}
        v_{\frac{i}{2}+1} & \mbox{ if } 0\leq i \leq 2(\lfloor \frac{n}{2}\rfloor -1) \mbox{ and } i \mbox{ is even } \\
        v_{\lfloor \frac{i}{2k} \rfloor + \lfloor \frac{n}{2} \rfloor +1} & \mbox{ if } 1\leq i \leq 2k\lfloor \frac{\lceil \frac{n}{2} \rceil}{k}\rfloor -1 \mbox{ and } i \mbox{ is odd }\\
        v_{\lfloor \frac{i}{2k} \rfloor + \lfloor \frac{n}{2} \rfloor} & \mbox{ otherwise. }\\
    \end{array}
\right.
$

\bigbreak

For each label $l$, we rename the $L$ vertices that have label $l$, following the increasing order of their number, into $l_0,l_1,...,l_L$, and we associate to the labeling the set of permutations $\sigma=\{\sigma_j, 1\leq j\leq k\}$:
\\$
\left\{
    \begin{array}{ll}
        \sigma_j(l_i)=l_{i+1\mod L} & \mbox{ for all $l$ }\\
        \sigma_j(v_i)=v_{i} & \mbox{ if } i \mbox{ is even. }\\
    \end{array}
\right.$

\bigbreak

For any $j$, a vertex and its image by $\sigma_j$ obviously have the same label. Plus, since for all $i$ with $0\leq i\leq n-1$ and $i$ even, $v_i$ is a fixed points, and for all $i$ odd and, $j\neq j'$, $\sigma_j(v_i)\neq \sigma_{j'}(v_i)$, $\sigma$ is a packing. 

\bigbreak

An example of the construction is given in the following figure, for $k=3$ and $n=7$:
 
\vspace{1cm}
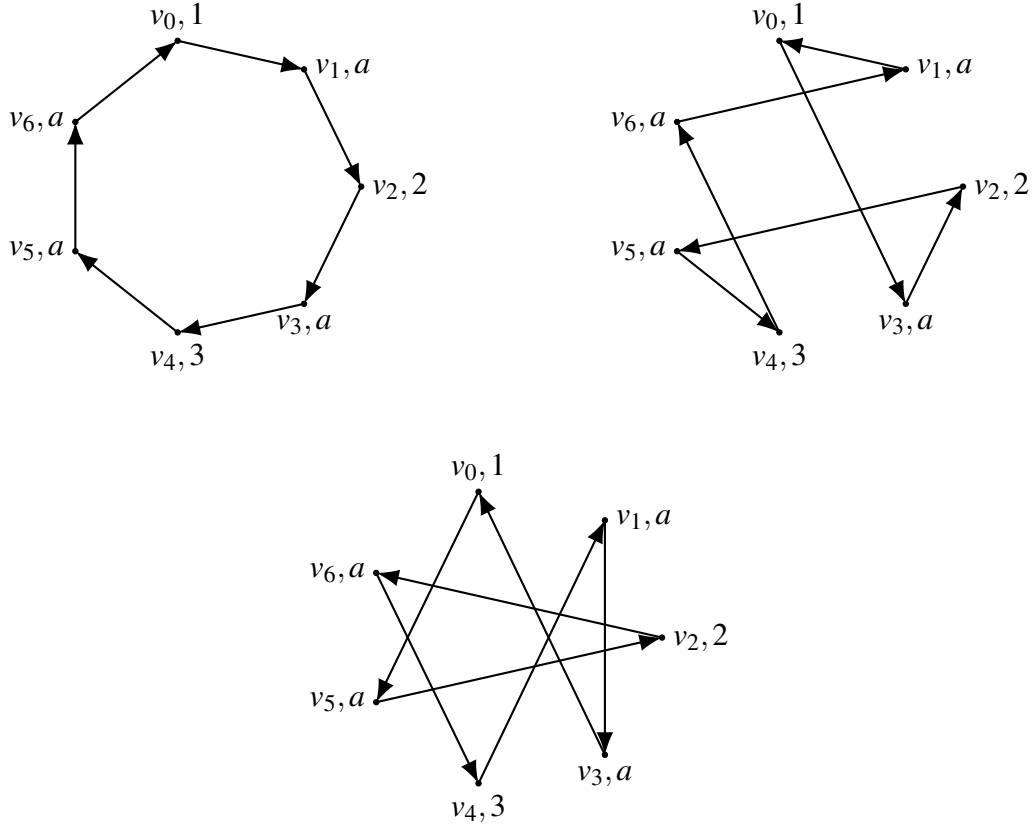
\begin{figure}[H]
	\centering
	\begin{tikzpicture}[scale=0.2]
    
    %Graphs
    
    \filldraw (10,0) circle (5pt);
    \draw (10,0) node[right]{$v_2,2$};
    \filldraw (6.2,7.8) circle (5pt);
    \draw (6.2,7.8) node[right]{$v_1,a$};
    \filldraw (-2.2,9.7) circle (5pt);
    \draw (-2.2,9.7) node[above]{$v_0,1$};
    \filldraw (-9,4.3) circle (5pt);
    \draw (-9,4.3) node[left]{$v_6,a$};
    \filldraw (-9,-4.3) circle (5pt);
    \draw (-9,-4.3) node[left]{$v_5,a$};
    \filldraw (-2.2,-9.7) circle (5pt);
    \draw (-2.2,-9.7) node[below]{$v_4,3$};
    \filldraw (6.2,-7.8) circle (5pt);
    \draw (6.2,-7.8) node[below]{$v_3,a$};  
    \draw[thick,-{Latex[length=3mm]}] (37.8,9.7) -- (46.2,-7.8);
    \draw[thick,-{Latex[length=3mm]}] (46.2,-7.8) -- (50,0);
    \draw[thick,-{Latex[length=3mm]}] (50,0) -- (31,-4.3);
    \draw[thick,-{Latex[length=3mm]}] (31,-4.3) -- (37.8,-9.7);
    \draw[thick,-{Latex[length=3mm]}] (37.8,-9.7) -- (31,4.3);
    \draw[thick,-{Latex[length=3mm]}] (31,4.3) -- (46.2,7.8);
    \draw[thick,-{Latex[length=3mm]}] (46.2,7.8) -- (37.8,9.7);
    
    \filldraw (50,0) circle (5pt);
    \draw (50,0) node[right]{$v_2,2$};
    \filldraw (46.2,7.8) circle (5pt);
    \draw (46.2,7.8) node[right]{$v_1,a$};
    \filldraw (37.8,9.7) circle (5pt);
    \draw (37.8,9.7) node[above]{$v_0,1$};
    \filldraw (31,4.3) circle (5pt);
    \draw (31,4.3) node[left]{$v_6,a$};
    \filldraw (31,-4.3) circle (5pt);
    \draw (31,-4.3) node[left]{$v_5,a$};
    \filldraw (37.8,-9.7) circle (5pt);
    \draw (37.8,-9.7) node[below]{$v_4,3$};
    \filldraw (46.2,-7.8) circle (5pt);
    \draw (46.2,-7.8) node[below]{$v_3,a$};  
    \draw[thick,-{Latex[length=3mm]}] (-2.2,9.7) -- (6.2,7.8);
    \draw[thick,-{Latex[length=3mm]}] (6.2,7.8) -- (10,0);
    \draw[thick,-{Latex[length=3mm]}] (10,0) -- (6.2,-7.8);
    \draw[thick,-{Latex[length=3mm]}] (6.2,-7.8) -- (-2.2,-9.7);
    \draw[thick,-{Latex[length=3mm]}] (-2.2,-9.7) -- (-9,-4.3);
    \draw[thick,-{Latex[length=3mm]}] (-9,-4.3) -- (-9,4.3);
    \draw[thick,-{Latex[length=3mm]}] (-9,4.3) -- (-2.2,9.7);
    
   \filldraw (30,-30) circle (5pt);
    \draw (30,-30) node[right]{$v_2,2$};
    \filldraw (26.2,-22.2) circle (5pt);
    \draw (26.2,-22.2) node[right]{$v_1,a$};
    \filldraw (17.8,-20.3) circle (5pt);
    \draw (17.8,-20.3) node[above]{$v_0,1$};
    \filldraw (11,-25.7) circle (5pt);
    \draw (11,-25.7) node[left]{$v_6,a$};
    \filldraw (11,-34.3) circle (5pt);
    \draw (11,-34.3) node[left]{$v_5,a$};
    \filldraw (17.8,-39.7) circle (5pt);
    \draw (17.8,-39.7) node[below]{$v_4,3$};
    \filldraw (26.2,-37.8) circle (5pt);
    \draw (26.2,-37.8) node[below]{$v_3,a$};  
    \draw[thick,-{Latex[length=3mm]}] (17.8,-20.3) -- (11,-34.3);
    \draw[thick,-{Latex[length=3mm]}] (11,-34.3) -- (30,-30);
    \draw[thick,-{Latex[length=3mm]}] (30,-30) -- (11,-25.7);
    \draw[thick,-{Latex[length=3mm]}] (11,-25.7) -- (17.8,-39.7);
    \draw[thick,-{Latex[length=3mm]}] (17.8,-39.7) -- (26.2,-22.2);
    \draw[thick,-{Latex[length=3mm]}] (26.2,-22.2) -- (26.2,-37.8);
    \draw[thick,-{Latex[length=3mm]}] (26.2,-37.8) -- (17.8,-20.3);
    
	\end{tikzpicture}
    \caption{A $4$-labeled packing of $k=3$ copies of $\vec{C_7}$.}
	\label{k2_circuits}
\end{figure}
\end{proof}

The case where $k+1\leq n\leq 2k-1$ is more complicated. We present here two theorems for this case, one for $k$ even, the second for $k$ odd:

\begin{theorem} For every circuit $\overrightarrow{C_n}$ of order $n=k+x$, $k\geq 2$, $1\leq x \leq k-1$, $k$ even, $(x,n) \not \in \{(1,4),(1,6)\}$, we have $\lambda^k(\overrightarrow{C_n}) \geq x+1$.
\end{theorem}
\begin{proof}
We give the associated construction, inspired by Walecki's \cite{alspach2008wonderful}. Let $\{v_0,v_1,...,v_{n-1}\}$ be the vertices of $\overrightarrow{C_n}$. We label them with $f$:

$f(v_i)=\left\{
    \begin{array}{ll}
        i-k+1 & \mbox{ if } i\geq k\\
        a & \mbox{ otherwise. }\\
    \end{array}
\right.
$

\bigbreak

We associate to it the set of permutations $\sigma=\{\sigma_j, 1\leq j\leq k\}$:
\\$\sigma_j(v_i)=
\left\{
    \begin{array}{ll}
    	v_{i} & \mbox{ if } i\geq k\\
        v_{i+j-1\mod k} & \mbox{ otherwise. }\\
    \end{array}
\right.$

\bigbreak

We partition the set of vertices $V(\overrightarrow{C_n})$ into a set $F$ containing all the fixed points, and a set $V_a$  containing all the vertices of label $a$. We partition the set of arcs $E(\overrightarrow{C_n})$ into two sets $E_F$ and $E_{a,a}$ defined as follows:
\\$E_F=\{(v_{i-1},v_{k+i})\cup (v_{k+i},v_{k-i}), 1\leq i\leq x-1\}\cup \{(v_k,v_0),(v_{\frac{k}{2}},v_k)\}$
\\$E_{a,a}=\{(v_{k-i},v_{i}), 1\leq i\leq \frac{k}{2}-1\}\cup \{(v_i,v_{k-1-i}), x-1\leq i\leq \frac{k}{2}-1\}$.

\bigbreak

The obtained graph $G$ is isomorphic to $\overrightarrow{C_n}$. Indeed, for $x=1$, if $H=(V(H)=\{v_1,v_2,...,v_n\},E(H)=\{(v_i,v_{i+1 \mod n}),1\leq i\leq n\})$, we have $H\simeq \overrightarrow{C_n}$, and $G=\tau(H)$, with:

$ \tau(v_i)=
\left\{
    \begin{array}{ll}
        v_{\frac{i}{2}} & \mbox{ if } i<k, i \equiv 0\mod 2\\
        v_{k-\frac{i+1}{2}} & \mbox{ if } i<k, i \equiv 1\mod 2\\
        v_{i} & \mbox{ if } i=k.\\
    \end{array}
\right.$

For $x>1$, the extra fixed point are inserted in some of the edges of $\overrightarrow{C_{k+1}}$, so that $G \simeq \overrightarrow{C_n}$.

\bigbreak

For all $v\in F$, since $\sigma_j(v)=v$, we know that $f(\sigma_j(v))=f(v)=f(\sigma_{j'}(v))$. Since $i+j-1\mod k<k$, for any $v\not \in F$, $f(\sigma_j(v))=a=f(\sigma_{j'}(v))$, and $f$ is a valid labeled packing with respect to $\sigma$.

\bigbreak

Thus, in particular, for all $j,j', \sigma_j^*(E_F)\cap \sigma_{j'}^*(E_{a,a})=\emptyset$.

\bigbreak

For all $\{u,v\},\{u',v'\}\in E_{a,a}$ , $\sigma_j(\{u,v\})=\sigma_{j'}(\{u',v'\})$ implies $j=j'$.
\\Indeed, for $1\leq i\leq \frac{k}{2}-1$, if $\sigma_j^*((v_{k-i},v_{i}))=\sigma_{j'}^*((v_{k-i'},v_{i'}))$, then:

\bigbreak

$\left\{
    \begin{array}{ll}
        i+j-1=i'+j'-1\mod k\\
        k-i+j-1=k-i'+j'-1\mod k.\\
    \end{array}
\right.$

\bigbreak

By adding the two equations, since $k$ is even, and $1\leq j,j'\leq k$, we get $j=j'$, or $j'=j+\frac{k}{2}$. But injecting $j'=j+\frac{k}{2}$ into the first equation gives $i\equiv i'+\frac{k}{2}\mod k$, which is a contradiction, as $1\leq i'\leq \frac{k}{2}-1$. 

\bigbreak

Similarly, for $x-1\leq i,i'\leq k-1$, if $\sigma_j^*((v_i,v_{k-1-i}))=\sigma_{j'}^*((v_{i'},v_{k-1-i'}))$, $j=j'$. 

\bigbreak

For $1\leq i\leq k-1$ and $x-1\leq i'\leq k-1$ , $\sigma_j^*((v_{k-i},v_i))=\sigma_{j'}^*((v_{i'},v_{k-1-i'}))$ is impossible, as the sum of both equations gives an equality between an odd number and an even number.

\bigbreak

Thus, for all $j\neq j', \sigma_j^*(E_{a,a})\cap \sigma_{j'}^*(E_{a,a})=\emptyset$.

\bigbreak

To show that $\sigma_j^*(E_F)\cap \sigma_{j'}^*(E_F)=\emptyset$, since each edge of $E_F$ contains one fixed point, we only have to show that its ingoing neighbor $u$ and outgoing neighbor $v$ verify $\sigma_j(u)\neq \sigma_{j'}(u)$ and $\sigma_j(v)\neq \sigma_{j'}(v)$. But since $i+j-1\mod k=i+j'-1\mod k$ implies $j=j'$, for all $v\not \in F$ and $j\neq j'$, $\sigma_j(v) \neq \sigma_{j'}(v)$.

\bigbreak

An example of the construction is given in the following figure, for $k=4$ and $x=3$:
\vspace{1cm}
\begin{figure}[H]
	\centering
	\begin{tikzpicture}[scale=0.2]
    
    %Graphs
    
    \filldraw (10,0) circle (5pt);
    \draw (10,0) node[right]{$v_2,a$};
    \filldraw (6.2,7.8) circle (5pt);
    \draw (6.2,7.8) node[right]{$v_1,a$};
    \filldraw (-2.2,9.7) circle (5pt);
    \draw (-2.2,9.7) node[above]{$v_0,a$};
    \filldraw (-9,4.3) circle (5pt);
    \draw (-9,4.3) node[left]{$v_6,3$};
    \filldraw (-9,-4.3) circle (5pt);
    \draw (-9,-4.3) node[left]{$v_5,2$};
    \filldraw (-2.2,-9.7) circle (5pt);
    \draw (-2.2,-9.7) node[below]{$v_4,1$};
    \filldraw (6.2,-7.8) circle (5pt);
    \draw (6.2,-7.8) node[below]{$v_3,a$};  
    \draw[thick,-{Latex[length=3mm]}] (-2.2,9.7) -- (-9,-4.3);
    \draw[thick,-{Latex[length=3mm]}] (-9,-4.3) -- (6.2,-7.8);
    \draw[thick,-{Latex[length=3mm]}] (6.2,-7.8) -- (6.2,7.8);
    \draw[thick,-{Latex[length=3mm]}] (6.2,7.8) -- (-9,4.3);
    \draw[thick,-{Latex[length=3mm]}] (-9,4.3) -- (10,0);
    \draw[thick,-{Latex[length=3mm]}] (10,0) -- (-2.2,-9.7);
    \draw[thick,-{Latex[length=3mm]}] (-2.2,-9.7) -- (-2.2,9.7);
    
    \filldraw (50,0) circle (5pt);
    \draw (50,0) node[right]{$v_2,a$};
    \filldraw (46.2,7.8) circle (5pt);
    \draw (46.2,7.8) node[right]{$v_1,a$};
    \filldraw (37.8,9.7) circle (5pt);
    \draw (37.8,9.7) node[above]{$v_0,a$};
    \filldraw (31,4.3) circle (5pt);
    \draw (31,4.3) node[left]{$v_6,3$};
    \filldraw (31,-4.3) circle (5pt);
    \draw (31,-4.3) node[left]{$v_5,2$};
    \filldraw (37.8,-9.7) circle (5pt);
    \draw (37.8,-9.7) node[below]{$v_4,1$};
    \filldraw (46.2,-7.8) circle (5pt);
    \draw (46.2,-7.8) node[below]{$v_3,a$};  
    \draw[thick,-{Latex[length=3mm]}] (46.2,7.8) -- (31,-4.3);
    \draw[thick,-{Latex[length=3mm]}] (31,-4.3) -- (37.8,9.7);
    \draw[thick,-{Latex[length=3mm]}] (37.8,9.7) -- (50,0);
    \draw[thick,-{Latex[length=3mm]}] (50,0) -- (31,4.3);
    \draw[thick,-{Latex[length=3mm]}] (31,4.3) -- (46.2,-7.8);
    \draw[thick,-{Latex[length=3mm]}] (46.2,-7.8) -- (37.8,-9.7);
    \draw[thick,-{Latex[length=3mm]}] (37.8,-9.7) -- (46.2,7.8);
    
   \filldraw (10,-30) circle (5pt);
    \draw (10,-30) node[right]{$v_2,a$};
    \filldraw (6.2,-22.2) circle (5pt);
    \draw (6.2,-22.2) node[right]{$v_1,a$};
    \filldraw (-2.2,-20.3) circle (5pt);
    \draw (-2.2,-20.3) node[above]{$v_0,a$};
    \filldraw (-9,-25.7) circle (5pt);
    \draw (-9,-25.7) node[left]{$v_6,3$};
    \filldraw (-9,-34.3) circle (5pt);
    \draw (-9,-34.3) node[left]{$v_5,2$};
    \filldraw (-2.2,-39.7) circle (5pt);
    \draw (-2.2,-39.7) node[below]{$v_4,1$};
    \filldraw (6.2,-37.8) circle (5pt);
    \draw (6.2,-37.8) node[below]{$v_3,a$};  
    \draw[thick,-{Latex[length=3mm]}] (10,-30) -- (-9,-34.3);
    \draw[thick,-{Latex[length=3mm]}] (-9,-34.3) -- (6.2,-22.2);
    \draw[thick,-{Latex[length=3mm]}] (6.2,-22.2) -- (6.2,-37.8);
    \draw[thick,-{Latex[length=3mm]}] (6.2,-37.8) -- (-9,-25.7);
    \draw[thick,-{Latex[length=3mm]}] (-9,-25.7) -- (-2.2,-20.3);
    \draw[thick,-{Latex[length=3mm]}] (-2.2,-20.3) -- (-2.2,-39.7);
    \draw[thick,-{Latex[length=3mm]}] (-2.2,-39.7) -- (10,-30);
    
    \filldraw (50,-30) circle (5pt);
    \draw (50,-30) node[right]{$v_2,a$};
    \filldraw (46.2,-22.2) circle (5pt);
    \draw (46.2,-22.2) node[right]{$v_1,a$};
    \filldraw (37.8,-20.3) circle (5pt);
    \draw (37.8,-20.3) node[above]{$v_0,a$};
    \filldraw (31,-25.7) circle (5pt);
    \draw (31,-25.7) node[left]{$v_6,3$};
    \filldraw (31,-34.3) circle (5pt);
    \draw (31,-34.3) node[left]{$v_5,2$};
    \filldraw (37.8,-39.7) circle (5pt);
    \draw (37.8,-39.7) node[below]{$v_4,1$};
    \filldraw (46.2,-37.8) circle (5pt);
    \draw (46.2,-37.8) node[right]{$v_3,a$}; 
    \draw[thick,-{Latex[length=3mm]}] (46.2,-37.8) -- (31,-34.3);
    \draw[thick,-{Latex[length=3mm]}] (31,-34.3) -- (50,-30);
    \draw[thick,-{Latex[length=3mm]}] (50,-30) -- (37.8,-20.3);
    \draw[thick,-{Latex[length=3mm]}] (37.8,-20.3) -- (31,-25.7);
    \draw[thick,-{Latex[length=3mm]}] (31,-25.7) -- (46.2,-22.2);
    \draw[thick,-{Latex[length=3mm]}] (46.2,-22.2) -- (37.8,-39.7);
    \draw[thick,-{Latex[length=3mm]}] (37.8,-39.7) -- (46.2,-37.8);
    
	\end{tikzpicture}
    \caption{A $4$-labeled packing of $k=4$ copies of $\vec{C_7}$.}
	\label{k2_circuits_2}
\end{figure}
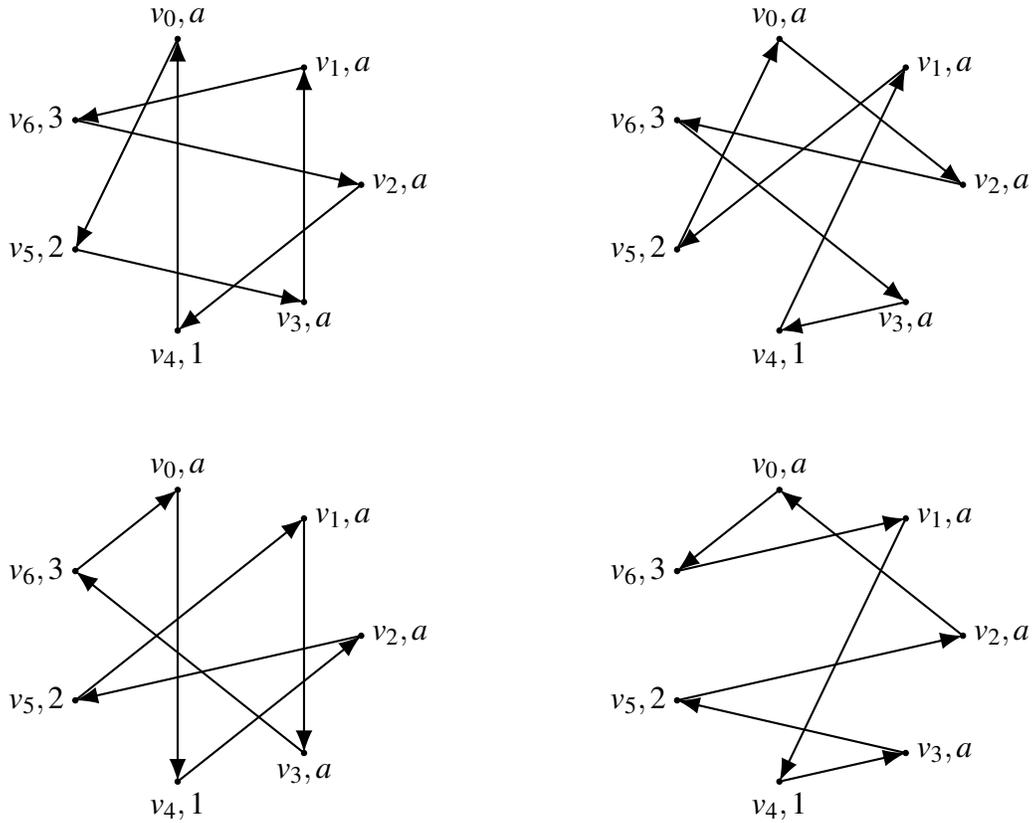
This concludes the proof of the theorem.
\end{proof}

The following result, for the case where $k$ is odd, can be deduced from the previous one:

\begin{theorem} For every circuit $\overrightarrow{C_n}$ of order $n=k+x$, $k\geq 2$, $2\leq x \leq k-1$, $k$ odd, $(x,n) \not \in \{(1,4),(1,6)\}$, we have $\lambda^k(\overrightarrow{C_n}) \geq x$.
\end{theorem} 
\begin{proof}
We have $n=(k+1)+(x-1)$ with $k+1$ even, and $x-1\geq 1$. Thus, by Theorem\ref{borne_sup_circuits}, there exists a $x$-labeled packing of $k+1$ copies of $\overrightarrow{C_n}$.
\end{proof}

\noindent The next results are interested in finding an upper bound. The proofs are not given as they are similar to the ones for cycles.

\bigbreak

The first result is the analogue of Lemma \ref{2klemma} for the case of circuits. The proof is identical, with the exception that we only count the outgoing neighbors and outgoing arcs of the vertices that have one of the labels that are represented by $q$ vertices. 

\begin{lemma}  Let $\overrightarrow{C_n}$ be the circuit of order $n=k+x$, $k\geq 2$, $1\leq x \leq k-1$, $(x,n) \not \in \{(1,4),(1,6)\}$. Let $f$ be a $p$-labeling of $k$ copies of $\overrightarrow{C_n}$. Let $q=\min \limits_{i\in \{1,...,p\}}|\{v\in V(C_n) \mbox{ such that } f(v)=i\}|$ be the minimum, over all $p$ labels, of the number of vertices that have this label. Then, there exists in $\overrightarrow{C_n}$ a set of at least $k$ vertices that are associated, together, to at most $q$ labels.
\label{2klemma_circuits}
\end{lemma}

This lemma leads to the two following theorems:

\begin{theorem}  Let $\overrightarrow{C_n}$ be the circuit of order $n=k+x$, $k\geq 2$, $1\leq x \leq k-1$, $(x,n) \not \in \{(1,4),(1,6)\}$. Let $f$ be a $p$-labeling of $k$ copies of $\overrightarrow{C_n}$. Let $q=\min \limits_{i\in \{1,...,p\}}|\{v\in V(\overrightarrow{C_n}) : f(v)=i\}|$ be the minimum, over all $p$ labels, of the number of vertices that have this label. If $q\leq x$, then $p \leq x+2$.
\label{qtheorem_circuits}
\end{theorem} 

As it was the case for cycles, the previous theorem gives conditions on the considered packing, while we want conditions on $k$ and $x$ only. Thus, we give the following result:

\begin{theorem}  Let $\overrightarrow{C_n}$ be the circuit of order $n=k+x$, $k\geq 2$, $1\leq x \leq k-1$, $(x,n) \not \in \{(1,4),(1,6)\}$. If $x\geq \sqrt{k-1}$, then $\lambda^k(\overrightarrow{C_n}) \leq x+2$.
\label{xtheorem_circuits}
\end{theorem} 

Finally, considering necessary conditions on the number of arcs between each pair of labels gives an analogue of Theorem \ref{max_nb_edges} for circuits:

\begin{theorem} Let $\overrightarrow{C_n}$ be the circuit of order $n=k+x$, $(x,n) \not \in \{(1,4),(1,6)\}$. For $p \in \mathbb{N}^*$, if $p\leq \lambda^k(\overrightarrow{C_n})$,  there exists a partition $n_1,n_2,...,n_p$ of $n$ into $p$ parts such that $\sum \limits_{i=1}^{p} \lfloor \frac{n_i(n_i-1)}{k} \rfloor + \sum \limits_{i=1}^{p} \sum \limits_{j=i+1}^{p} \lfloor \frac{2n_in_j}{k} \rfloor \geq n$.
\label{max_nb_edges_circuits}
\end{theorem}

\bibliographystyle{abbrv}
\bibliography{biblio}

\end{document}